\documentclass[11pt,a4paper,reqno]{amsart} 
\usepackage[applemac]{inputenc}
\usepackage{amsmath}
\usepackage{amssymb}
\usepackage{euscript}
\usepackage{mathrsfs}
\usepackage{wasysym}
\usepackage[all]{xy}
\usepackage{graphicx}
\usepackage{color}
\usepackage{multirow}
\usepackage{pifont}
\usepackage[table]{xcolor}
\usepackage[colorlinks=true,linktocpage=true,pagebackref=false, urlcolor=black, citecolor=black,linkcolor=black]{hyperref}

\newtheorem{thm}{Theorem}[section]

\newtheorem{lem}[thm]{Lemma}
\newtheorem{prop}[thm]{Proposition}

\newtheorem{cor}[thm]{Corollary}

\newtheorem{defn}[thm]{Definition}

\newtheorem{remark}[thm]{Remark}
\newtheorem{remarks}[thm]{Remarks}
\newtheorem{example}[thm]{Example}
\newtheorem{examples}[thm]{Examples}

\numberwithin{equation}{section}

 \newcommand{\N}{{\mathbb N}}
\newcommand{\Z}{{\mathbb Z}} \newcommand{\R}{{\mathbb R}}
\newcommand{\Q}{{\mathbb Q}}

\newcommand{\Ff}{{\EuScript F}}
\newcommand{\Jj}{{\EuScript I}}

\newcommand{\Rr}{{\EuScript R}}
\newcommand{\an}{{\EuScript O}}

\newcommand{\Int}{\operatorname{Int}}
\newcommand{\cl}{\operatorname{Cl}}

\newcommand{\supp}{\operatorname{supp}}

\newcommand{\id}{\operatorname{id}}

\newcommand{\x}{{\tt x}} \newcommand{\y}{{\tt y}} 
\newcommand{\z}{{\tt z}} \renewcommand{\t}{{\tt t}}

\newcommand{\veps}{\varepsilon}

\newcommand{\ol}{\overline}

\numberwithin{equation}{section}

\begin{document}
\title[On the set of local extrema of a subanalytic function]{On the set of local extrema of a subanalytic function}

\author{Jos\'e F. Fernando}
\address{Departamento de \'Algebra, Facultad de Ciencias Matem\'aticas, Universidad Complutense de Madrid, 28040 MADRID (SPAIN)}
\email{josefer@mat.ucm.es}

\date{31/10/2017}
\dedicatory{Dedicated to my beloved friends Francesca Acquistapace and Fabrizio Broglia\\ in occasion of their 70th birthdays}

\thanks{Author supported by Spanish GRAYAS MTM2014-55565-P and Grupos UCM 910444.}

\maketitle

\begin{abstract} 
Let ${\mathfrak F}$ be a category of subanalytic subsets of real analytic manifolds that is closed under basic set-theoretical operations (locally finite unions, difference and product) and basic topological operations (taking connected components and closures). Let $M$ be a real analytic manifold and denote ${\mathfrak F}(M)$ the family of the subsets of $M$ that belong to the category ${\mathfrak F}$. Let $f:X\to\R$ be a subanalytic function on a subset $X\in{\mathfrak F}(M)$ such that the inverse image under $f$ of each interval of $\R$ belongs to ${\mathfrak F}(M)$. Let ${\rm Max}(f)$ be the set of local maxima of $f$ and consider its level sets ${\rm Max}_\lambda(f):={\rm Max}(f)\cap\{f=\lambda\}=\{f=\lambda\}\setminus\cl(\{f>\lambda\})$ for each $\lambda\in\R$. In this work we show that if $f$ is continuous, then ${\rm Max}(f)=\bigsqcup_{\lambda\in\R}{\rm Max}_\lambda(f)\in{\mathfrak F}(M)$ if and only if the family $\{{\rm Max}_\lambda(f)\}_{\lambda\in\R}$ is locally finite in $M$. If we erase continuity condition, there exist subanalytic functions $f:X\to M$ such that ${\rm Max}(f)\in{\mathfrak F}(M)$, but the family $\{{\rm Max}_\lambda(f)\}_{\lambda\in\R}$ is not locally finite in $M$ or such that ${\rm Max}(f)$ is connected but it is not even subanalytic. We show in addition that if ${\mathfrak F}$ is the category of subanalytic sets and $f:X\to\R$ is a (non-necessarily continuous) subanalytic map $f$ that maps relatively compact subsets of $M$ contained in $X$ to bounded subsets of $\R$, then ${\rm Max}(f)\in{\mathfrak F}(M)$ and the family $\{{\rm Max}_\lambda(f)\}_{\lambda\in\R}$ is locally finite in $M$. An example of this type of functions are continuous subanalytic functions on closed subanalytic subsets of $M$. The previous results imply that if ${\mathfrak F}$ is either the category of semianalytic sets or the category of $C$-semianalytic sets and $f$ is the restriction to an ${\mathfrak F}$-subset of $M$ of an analytic function on $M$, then the family $\{{\rm Max}_\lambda(f)\}_{\lambda\in\R}$ is locally finite in $M$ and ${\rm Max}(f)=\bigsqcup_{\lambda\in\R}{\rm Max}_\lambda(f)\in{\mathfrak F}(M)$. We also show that if the category ${\mathfrak F}$ contains the intersections of algebraic sets with real analytic submanifolds and $X\in{\mathfrak F}(M)$ is not closed in $M$, then there exists a continuous subanalytic function $f:X\to\R$ with graph belonging to ${\mathfrak F}(M\times\R)$ such that inverse images under $f$ of the intervals of $\R$ belong to ${\mathfrak F}(M)$ but ${\rm Max}(f)$ does not belong to ${\mathfrak F}(M)$.

As subanalytic sets are locally connected, the set of non-openness points of a continuous subanalytic function $f:X\to\R$ coincides with the set of local extrema ${\rm Extr}(f):={\rm Max}(f)\cup{\rm Min}(f)$. This means that if $f:X\to\R$ is a continuous subanalytic function defined on a closed set $X\in{\mathfrak F}(M)$ such that the inverse image under $f$ of each interval of $\R$ belongs to ${\mathfrak F}(M)$, then the set ${\rm Op}(f)$ of openness points of $f$ belongs to ${\mathfrak F}(M)$. Again the closedness of $X$ in $M$ is crucial to guarantee that ${\rm Op}(f)$ belongs to ${\mathfrak F}(M)$.

The type of results stated above are straightforward if ${\mathfrak F}$ is an o-minimal structure of subanalytic sets. However, the proof of the previous results requires further work for a category ${\mathfrak F}$ of subanalytic sets that does not constitute an o-minimal structure. 
\end{abstract}

\section{Introduction}\label{s1}

Let $X$ be a topological space and let $f:X\to\R$ be a real function. We say that $f$ has a \em local maximum \em (resp. \em local minimum\em) at $x_0\in X$ if there exists an open neighborhood $W\subset X$ of $x_0$ such that $f(x_0)\geq f(x)$ (resp. $f(x_0)\leq f(x)$) for each $x\in W$. If $f(x_0)>f(x)$ (resp. $f(x_0)<f(x)$) for each $x\in W\setminus\{x_0\}$, we say that $f$ has a \em strict local maximum \em (resp. \em strict local minimum\em) at $x_0$. We use \em local extrema \em to refer indistinctly to local maxima or local minima. A point $x_0\in X$ is a \em global maximum \em (resp. \em global minimum\em) of $f$ if $f(x_0)\geq f(x)$ (resp. $f(x_0)\leq f(x)$) for each $x\in X$. If $f(x_0)>f(x)$ (resp. $f(x_0)<f(x)$) for each $x\in X\setminus\{x_0\}$, we say that $f$ has a \em strict global maximum \em (resp. \em strict global minimum\em) at $x_0$. We denote ${\rm Max}(f)$ the set of local maxima of $f$ whereas ${\rm Min}(f)$ refers to the set of local minima of $f$. Observe that ${\rm Min}(f)={\rm Max}(-f)$, so it is enough to study the properties of the set of local maxima to understand both sets. To lighten the exposition we will state the results only for the set of local maxima. The union ${\rm Extr}(f):={\rm Max}(f)\cup{\rm Min}(f)={\rm Max}(f)\cup{\rm Max}(-f)$ is the set of local extrema of $f$. In the following we denote ${\rm Max}_\lambda(f):={\rm Max}(f)\cap\{f=\lambda\}$ for each $\lambda\in\R$. For each $\lambda\in\R$ the level set ${\rm Max}_\lambda(f)$ coincides with the set $\{f=\lambda\}\setminus\cl(\{f>\lambda\})$, see Lemma \ref{altdes}. We have ${\rm Max}(f)=\bigsqcup_{\lambda\in\R}{\rm Max}_\lambda(f)$ and ${\rm Max}_\lambda(f)\neq\varnothing$ if and only if $\lambda\in f({\rm Max}(f))$. We use the symbol $\sqcup$ to denote pairwise disjoint unions.

Let $X,Y$ be topological spaces and let $S\subset X$ be a subset. We use the following notation: let $\cl(S)$, $\Int(S)$ and $\partial(S)(:=\cl(S)\setminus\Int(S))$ denote respectively the closure, interior and boundary of $S$ in $X$. A map $f:S\to Y$ is $X$-compact if $f(K\cap S)$ is a relatively compact subset of $Y$ for each compact subset $K\subset X$. A relevant example of $X$-compact maps is that of continuous maps $f:S\to Y$ on closed subsets $S$ of $X$. We say that a family ${\mathcal S}:=\{S_i\}_{i\in I}$ is \em locally isolated in $X$ \em if for each point $x\in X$ there exists an open neighborhood $U^x\subset X$ of $x$ such that $U^x$ meet at most one of the members of ${\mathcal S}$.

The study of local extrema of real functions of different types (continuous, differentiable, analytic, subanalytic, etc.) defined on several types of spaces (topological spaces, open subsets of affine spaces, real analytic manifolds, subanalytic sets, etc.) is long and rich and we refer the reader for instance to \cite{bpw,bgz,cv,fl,ko} for further information. The approach in these articles is mainly from the local viewpoint (and concerns the behavior of a function in a small neighborhood around a local extrema) but there is an important lack of information about the properties of the set of local extrema from a global viewpoint. In this paper we analyze for some categories ${\mathfrak F}$ of subanalytic sets, satisfying mild properties, the belonging to the category ${\mathfrak F}$ of the sets of local maxima, local minima and local extrema of the functions whose graphs belong to ${\mathfrak F}$ and satisfy some mild additional conditions. We focus mainly in the categories ${\mathfrak F}$ of subanalytic, semianalytic y $C$-semianalytic sets.

\subsection{Semianalytic, $C$-semianalytic and subanalytic sets and functions} 
Let $M$ be a real analytic manifold. A subset $S\subset M$ is \em semianalytic \em if each point of $M$ admits a neighborhood $U$ such that $S\cap U$ can be described as a finite Boolean combination of analytic equalities and inequalities, where the involved functions are analytic (and possibly only) on $U$. We say that $Z\subset M$ is $C$-analytic, if it is the common zero set of finitely many real analytic functions on $M$. A subset $S\subset M$ is a \em basic $C$-semianalytic set \em if it admits a description of the type
$$
S:=\{x\in M:\ f(x)=0,g_1(x)>0,\ldots,g_r(x)>0\}
$$
where the functions $f,g_i:M\to\R$ are analytic on $M$. We say that $S\subset M$ is a \em $C$-semianalytic set \em \cite{abf} if it satisfies one of the following equivalent conditions:
\begin{itemize}
\item[(1)] $S$ is the union of a countable locally finite family of basic $C$-semianalytic sets.
\item[(2)] For each point $x\in M$ there exists an open neighborhood $U^x\subset M$ such that $S\cap U^x$ is a finite union of basic $C$-semianalytic sets. 
\end{itemize}
Recall that $X\subset M$ is \em subanalytic \em if each point of $M$ admits a neighborhood $U$ such that $X\cap U$ is a projection of a relatively compact semianalytic set (that is, there exist a real analytic manifold $N$ and a relatively compact semianalytic subset $A$ of $M\times N$ such that $X\cap U=\pi(A)$, where $\pi:M\times N\to M$ is the projection onto the first factor). 

Semianalytic sets (and more generally subanalytic sets) were introduced by \L ojasiewicz in \cite{l,l1} and were developed later by many authors: Bierstone--Milman \cite{bm,bm2}, Hironaka \cite{hi1,hi2,hi3,hi4}, Gabrielov \cite{ga}, Hardt \cite{h1,h2}, Galbiati \cite{gal}, Paw\l ucki \cite{pa}, Denkowska \cite{de}, Stasica \cite{s}, Kurdyka \cite{k}, Parusi\'nski \cite{p}, Shiota \cite{sh} among others. These sets have many and wide applications in complex and real analytic geometry. Whereas the family of complex analytic sets is stable under proper holomorphic maps between complex analytic spaces (Remmert's Theorem \cite[VII.\S2.Thm.2]{n}), an analogous property does not hold in the real analytic setting. The image of a real analytic set under a proper real analytic map is not even in general a semianalytic set. This fact promoted the introduction of subanalytic sets by \L ojasiewicz \cite{l} in the 1960s. In fact, subanalytic sets are characterized as the images of semianalytic sets under proper analytic maps \cite{bm}.

In \cite{abf} we introduced, amalgamating the notions of $C$-analytic sets and semianalytic sets, the concept of $C$-semianalytic set. Our aim was to find a family of semianalytic sets `globally defined' in the sense of Cartan and Hironaka that enjoy a good behavior with respect to basic set-theoretical, topological and algebraic operations. In fact, in \cite[Thm. 1.5]{abf} we characterize subanalytic sets as images of basic $C$-semianalytic sets under proper analytic maps. 

Let $P$ be a property concerning $C$-semianalytic sets. We say that $P$ is a $C$-property if the set of points of a $C$-semianalytic set $X$ satisfying $P$ is a $C$-semianalytic set. For example, in \cite{abf} we showed that the set of points for which the dimension of the $C$-semianalytic set $X$ is a fixed integer $k$ is again a $C$-semianalytic set, that is, `to be a point of dimension $k$' is a $C$-property. We also proved in \cite{abf} that the set of points of non-coherence of a $C$-analytic set is $C$-semianalytic, that is, `to be a point of non-coherence' (or `to be a point of coherence') are $C$-properties. In this work we show that the set of local maxima, local minima and local extrema of the restriction to a $C$-semianalytic subset of a real analytic manifold $M$ of a real analytic function on $M$ is a $C$-semianalytic set (Corollary \ref{semicsemi}). This means that `to be a local maximum, a local minimum or a local extremum' of the restriction to a $C$-semianalytic set of a global analytic function are $C$-properties.

\subsubsection{Weak categories.} 
Let ${\mathfrak M}$ be the class of all real analytic manifolds. A \em weak category (of subanalytic sets) \em ${\mathfrak F}:=\{{\mathfrak F}(M)\}_{M\in{\mathfrak M}}$ is a collection of families ${\mathfrak F}(M)$ of subanalytic subsets of $M$ such that the following conditions are satisfied for each $M,N\in{\mathfrak M}$:
\begin{itemize}
\item[$\bullet$] $M\in{\mathfrak F}(M)$.
\item[$\bullet$] If $S_1,S_2\in{\mathfrak F}(M)$, then $S_1\setminus S_2\in{\mathfrak F}(M)$.
\item[$\bullet$] If $\{S_i\}_{i\in I}\subset{\mathfrak F}(M)$ is a locally finite family in $M$, then $\bigcup_{i\in I}S_i\in{\mathfrak F}(M)$.
\item[$\bullet$] If $S\in{\mathfrak F}(M)$ and $T\in{\mathfrak F}(N)$, then $S\times T\in{\mathfrak F}(M\times N)$.
\item[$\bullet$] If $S\in{\mathfrak F}(M)$, then its connected components and its closure $\cl(S)$ in $M$ belong to ${\mathfrak F}(M)$.
\end{itemize}
The previous properties guarantee that if $S,S_1,S_2\in{\mathfrak F}(M)$, then $\Int(S),\partial S,S_1\cap S_2\in{\mathfrak F}(M)$. 

We say that a weak category ${\mathfrak F}$ \em contains algebraic intersections \em if for each real analytic submanifold $M\subset\R^n$ where $n\geq1$, the intersection $M\cap X\in{\mathfrak F}(M)$ for each algebraic set $X\subset\R^n$. The conditions satisfied by a weak category ${\mathfrak F}$ guarantee that if ${\mathfrak F}$ contains algebraic intersections, it also \em contains semialgebraic intersections\em, that is, for each real analytic submanifold $M\subset\R^n$ where $n\geq1$ the intersection $M\cap S\in{\mathfrak F}$ for each semialgebraic set $S\subset\R^n$. If ${\mathfrak F}$ is either the category of subanalytic, semianalytic or $C$-semianalytic sets, ${\mathfrak F}$ is a weak category that contain algebraic intersections, as it contains $C$-analytic sets. This fact has further consequences: \em if $N$ is a closed analytic submanifold of $M$, then $N$ is a $C$-analytic subset of $M$ and\em 
\begin{equation}\label{elias}
{\mathfrak F}(N)=\{Y\in{\mathfrak F}(M):\ Y\subset N\}\subset{\mathfrak F}(M).
\end{equation}
\begin{proof}
By Cartan's Theorem $B$ the analytic functions on $N$ are the restrictions to $N$ of global analytic functions on $M$. Thus, if ${\mathfrak F}$ is either the category of semianalytic or $C$-semianalytic sets, \eqref{elias} holds. Assume next ${\mathfrak F}$ is the category of subanalytic sets. If $X\in{\mathfrak F}(N)$, there exist an analytic map $f:N'\to N$ where $N'\in{\mathfrak M}$ and a semianalytic set $S$ of $N'$ such that the restriction $f|_{\cl(S)}:\cl(S)\to N$ is proper and $f(S)=X$. As $N$ is closed in $M$, the restriction $f|_{\cl(S)}:\cl(S)\to M$ of the analytic map $f:N'\to N\subset M$ is also proper and $X\in{\mathfrak F}(M)$, so ${\mathfrak F}(N)\subset{\mathfrak F}(M)$. Conversely, if $Y\in{\mathfrak F}(M)$, there exist an analytic map $g:M'\to M$ where $M'\in{\mathfrak M}$ and a semianalytic set $T$ of $M'$ such that the restriction $g|_{\cl(T)}:\cl(T)\to M$ is proper and $g(T)=Y$. Let $h$ be an analytic equation of $N$ in $M$. Define $T':=\{h\circ g=0\}=g^{-1}(N)$, which is a $C$-analytic subset of $M'$. Observe that $g|_{\cl(T\cap T')}:\cl(T\cap T')\to N$ is proper and $g(T\cap T')=Y\cap N$, so $Y\cap N\in{\mathfrak F}(N)$.
\end{proof} 

Whitney's immersion theorem provides immersions as analytic submanifolds of Euclidean spaces for the real analytic manifolds. This fact makes that if ${\mathfrak F}$ is either the category of subanalytic, semianalytic or $C$-semianalytic sets, ${\mathfrak F}=\{{\mathfrak F}_m:={\mathfrak F}(\R^m)\}_{m\geq1}$ and if $M\subset\R^m$ is a closed analytic submanifold, ${\mathfrak F}(M)=\{X\in{\mathfrak F}(\R^m):\ X\subset M\}$. For a (non-necessarily immersed) real analytic manifold $M$ we consider an analytic immersion $\varphi:M\hookrightarrow\R^m$ such that $\varphi(M)$ is a closed analytic submanifold of $\R^m$ and it holds ${\mathfrak F}(M)=\{\varphi^{-1}(X):\ X\in{\mathfrak F}(\R^m)\text{ and } X\subset\varphi(M)\}$.

Given a subset $X\subset M$ and a real analytic manifold $N$ an \em ${\mathfrak F}$-map \em is a map $f:X\to N$ whose graph is an ${\mathfrak F}$-subset of $M\times N$. In case ${\mathfrak F}$ is the weak category of subanalytic sets, we say that $f$ is a \em subanalytic map \em and we proceed analogously with semianalytic and $C$-semianalytic categories.

For a careful study of more restrictive categories of subanalytic sets (\em analytic-geometric categories \em \cite{dm}) that satisfy stronger properties (for instance, the image under proper ${\mathfrak F}$-map of an ${\mathfrak F}$-set is an ${\mathfrak F}$-set) that assures the existence of Whitney's stratifications or subanalytic triangulations, we refer the reader to \cite{dm,sh}.

\subsection{Main results}
In this work we analyze the properties of the set ${\rm Max}(f)$ (resp. ${\rm Min}(f)$ and ${\rm Extr}(f)$) for a subanalytic function $f:X\to\R$ on a subanalytic subset $X$ of a real analytic manifold $M$. Our main results are the following.

\begin{thm}\label{main0}
Let ${\mathfrak F}$ be a weak category and let $f:X\to\R$ be a continuous subanalytic function on $X\in{\mathfrak F}(M)$ such that the inverse images under $f$ of intervals of $\R$ belong to ${\mathfrak F}(M)$. The following assertions are equivalent: 
\begin{itemize}
\item[(i)] ${\rm Max}(f)\in{\mathfrak F}(M)$. 
\item[(ii)] The family of the connected components of ${\rm Max}(f)$ is locally finite in $M$.
\item[(iii)] The family $\{{\rm Max}_\lambda(f)\}_{\lambda\in\R}$ is locally finite in $M$.
\item[(iv)] The family $\{{\rm Max}_\lambda(f)\}_{\lambda\in\R}$ is locally isolated in $M$.
\end{itemize}
\end{thm}
If we are under the hypotheses of Theorem \ref{main0}, the sets ${\rm Max}_\lambda(f)=\{f=\lambda\}\setminus\cl(\{f>\lambda\})$ belong to ${\mathfrak F}(M)$ (see the definition of weak category) whereas we can only assure that the connected components of ${\rm Max}(f)$ belong to ${\mathfrak F}(M)$ if ${\rm Max}(f)\in{\mathfrak F}(M)$. If we erase the continuity condition, we show in Examples \ref{ptd} (i) and \ref{ex1} that the previous result does not remain true. If ${\mathfrak F}$ is the weak category of subanalytic sets, we obtain the following.

\begin{thm}\label{main1}
Let $f:X\to\R$ be an $M$-compact subanalytic function. Then the inverse images under $f$ of intervals of $\R$ are subanalytic subsets of $M$ and the family $\{{\rm Max}_\lambda(f)\}_{\lambda\in\R}$ is locally finite in $M$. Consequently, each set ${\rm Max}_\lambda(f)=\{f=\lambda\}\setminus\cl(\{f>\lambda\})$ and ${\rm Max}(f)$ are subanalytic subsets of $M$.
\end{thm}

If $f:X\to\R$ is a continuous subanalytic function on a closed subanalytic subset of $M$, then $f$ is $M$-compact and we have the following lemma, which is a key result to prove Theorems \ref{main0} and \ref{main1}.

\begin{lem}\label{main2}
Let $f:X\to\R$ be a continuous subanalytic function on a closed subanalytic subset of $M$. Then the family $\{{\rm Max}_\lambda(f)\}_{\lambda\in\R}$ is locally finite in $M$ and each set ${\rm Max}_\lambda(f)$ and ${\rm Max}(f)$ are subanalytic subsets of $M$.
\end{lem} 

If ${\mathfrak F}$ is either the weak category of semianalytic or $C$-semianalytic sets we have the following.

\begin{cor}\label{semicsemi}
Let ${\mathfrak F}$ be either the weak category of semianalytic or $C$-semianalytic sets and let $f:X\to\R$ be the restriction to $X\in{\mathfrak F}(M)$ of an analytic function on $M$. Then the family $\{{\rm Max}_\lambda(f)\}_{\lambda\in\R}$ is locally finite in $M$ and ${\rm Max}(f)$ and each set ${\rm Max}_\lambda(f)$ belong to ${\mathfrak F}(M)$.
\end{cor}
\begin{proof}
As $f$ is the restriction to $X$ of an analytic function on $M$, it is an $M$-compact subanalytic function and the inverse images under $f$ of intervals of $\R$ belong to ${\mathfrak F}(M)$. In particular, the sets ${\rm Max}_\lambda(f)=\{f=\lambda\}\setminus\cl(\{f>\lambda\})\in{\mathfrak F}(M)$. By Theorem \ref{main1} the family $\{{\rm Max}_\lambda(f)\}_{\lambda\in\R}$ is locally finite in $M$ and ${\rm Max}(f)=\bigsqcup_{\lambda\in\R}{\rm Max}_\lambda(f)\in{\mathfrak F}(M)$, as required. 
\end{proof}

In order to prove the sharpness of our results we show the following.

\begin{prop}\label{sharp}
Let ${\mathfrak F}$ be a weak category that contains algebraic intersections and let $X\in{\mathfrak F}(M)$ be not closed in $M$. Then there exists a continuous ${\mathfrak F}$-function $f:X\to\R$ such that the inverse images under $f$ of intervals of $\R$ are subanalytic subsets of $M$ but the families $\{{\rm Max}_\lambda(f)\}_{\lambda\in\R}$, $\{{\rm Min}_\lambda(f)={\rm Max}_{-\lambda}(-f)\}_{\lambda\in\R}$ and $\{{\rm Extr}_\lambda(f)={\rm Max}_\lambda(f)\cup{\rm Max}_{-\lambda}(-f)\}_{\lambda\in\R}$ are not locally finite in $M$. Consequently, the sets ${\rm Max}_\lambda(f),{\rm Min}_\lambda(f),{\rm Extr}_\lambda(f)\in{\mathfrak F}(M)$ for each $\lambda\in\R$ whereas the sets ${\rm Max}(f)$, ${\rm Min}(f)$ and ${\rm Extr}(f)$ do not belong to ${\mathfrak F}(M)$.
\end{prop}

\subsection{Semialgebraic case and o-minimal structures}
A first case one analyzes concerns the semialgebraic setting and more general o-minimal structures. A \em semialgebraic set \em $S\subset\R^n$ is a subset of $\R^n$ that can be described as a finite Boolean combination of polynomial equalities and inequalities whereas a \em semialgebraic function \em $f:S\to\R$ is a function whose graph is a semialgebraic subset of $\R^{n+1}$. Tarski-Seidenberg's theorem states that the projection of a semialgebraic set is again semialgebraic. 

We recall next the definition of an o-minimal structure. 

\begin{defn}\label{ominstr}
An \em o-minimal structure \em on the field $\R$ of real numbers is a collection ${\mathfrak S}:=\{{\mathfrak S}_n\}_{n\in\N}$ of
families ${\mathfrak S}_n$ of subsets of $\R^n$ satisfying:
\begin{itemize}
\item[(1)] ${\mathfrak S}_n$ contains all the algebraic subsets of $\R^n$.
\item[(2)] ${\mathfrak S}_n$ is a Boolean algebra.
\item[(3)] If $A\in{\mathfrak S}_m$ and $B\in{\mathfrak S}_n$, then $A\times B\in{\mathfrak S}_{m+n}$.
\item[(4)] If $\pi:\R^n\times\R\to\R^n$ is the natural projection and $A\in{\mathfrak S}_{n+1}$, then $\pi(A)\in{\mathfrak S}_n$.
\item[(5)] ${\mathfrak S}_1$ consists precisely of all the finite unions of points and intervals of any type.
\end{itemize}
\end{defn}
The elements of ${\mathfrak S}_n$ are called \em definable subsets of $\R^n$ \em and a map is called \em definable \em if its graph is a definable set. The concept of o-minimal structure arose within the framework of Model Theory. Briefly, we fix a language $L$ of symbols that represent functions, relations and constants of $\R$, and that contain the symbols for the ordered field structure of $\R$. The \em atomic formulas \em of $L$ are those of the form $f_1(\x_1,\ldots,\x_n){\Rr}f_2(\x_1,\ldots,\x_n)$ where $f_1$ and $f_2$ are compositions of the functions in $L$ and $\Rr$ is a relation in $L$. A \em first order formula \em is written with a finite number of conjunctions, disjunctions, and universal or existencial quantifiers on some of the variables of the atomic formulas. An $L$-structure on $\R$ is an interpretation of the symbols in $L$ and the subsets described by the first order formula are called \emph{definable}. Then, we say that such an $L$-structure is o-minimal if every definable (possibly with parameters) subset of $\R$ is a finite union of intervals and points. 

As a consequence of Tarski-Seidenberg's theorem, semialgebraic sets constitute an o-minimal structure and in fact it is contained in each o-minimal structure on $\R$. The categories of semialgebraic sets and definable sets in o-minimal structures enjoy similar properties. For instance, each definable map is piecewise continuous \cite[Thm.3.2.11]{vdD1} and we have definable choice \cite[Prop.6.1.2]{vdD1}.

\begin{lem}\label{semiom}
Let $\mathfrak{S}$ be an o-minimal structure. Let $S\subset\R^n$ be a definable set and $f:S\to\R$ a definable function. Then the set ${\rm Max}(f)$ of local maxima of $f$ is a definable subset of $\R^n$. In addition, $f({\rm Max}(f))$ is a finite set.
\end{lem}
\begin{proof}
The set ${\rm Max}(f)$ can be described as the set of points $x\in S$ for which there exists $\veps>0$ satisfying $f(x)\geq f(y)$ for each $y\in S\cap\{\|\y-x\|<\veps\}$. Thus, ${\rm Max}(f)$ is clearly a definable subset of $\R^n$. 

Suppose next that $f({\rm Max}(f))$ is an infinite set. By property (5) in Definition \ref{ominstr} there exists an infinite open subinterval $I\subset f({\rm Max}(f))$. By definable choice there exists a definable map $\alpha:I\rightarrow{\rm Max}(f)$ such that $f(\alpha(t))=t$ for each $t\in I$ and by \cite[Thm.3.2.11]{vdD1} we may assume in addition (after shrinking $I$) that $\alpha$ is continuous. Fix $t_0\in I$. As $\alpha(t_0)$ is a local maximum, there exists an open neighborhood $U$ of $\alpha(t_0)$ such that $f(x)\leq f(\alpha(t_0))=t_0$ for each $x\in U$. As $\alpha$ is continuous, $\alpha^{-1}(U)$ is open and contains $t_0$, hence there exists $t_1\in\alpha^{-1}(U)$ such that $t_0<t_1$. But $\alpha(t_1)\in U$, so $t_1=f(\alpha(t_1))\leq t_0$, which is a contradiction. Thus, $f({\rm Max}(f))$ is a finite set, as required.
\end{proof}

By Lemma \ref{semiom} the family $\{{\rm Max}_\lambda(f)\}_{\lambda\in\R}$ is finite for each definable function $f$. Analogously, the sets of local minima, strict local maxima, strict local minima, global maxima, global minima, strict global maxima, strict global minima, local extrema, strict local extrema, global extrema and strict global extrema of $f$ are definable subsets of $\R^n$ and their images under $f$ are finite sets. 

In the subanalytic setting it is also possible to find a category that constitutes an o-minimal structure. A \em restricted analytic function in $n$-variables \em is a function $f:[-1,1]^n\to\R$ that admits an analytic continuation to an open neighborhood of $[-1,1]^n$ of $\R^n$. A \em global subanalytic subset \em of $\R^n$ is a subset $X\subset\R^n$ such that there exists a semialgebraic homeomorphism $g:\R^n\to(-1,1)^n$ satisfying that $g(X)$ is a subanalytic subset of $\R^n$. The collection of global subanalytic sets is precisely the collection of definable sets in the o-minimal structure $\R_{\rm an}$ generated by the set $\widetilde{\mathcal F}_{\rm an}$ of restricted analytic functions \cite{dd,w}. Thus, if $X\subset\R^n$ is a global subanalytic set and $f:X\to\R$ is a definable function of the o-minimal structure $\R_{\rm an}$, the sets of local minima, strict local maxima, strict local minima, global maxima, global minima, strict global maxima, strict global minima, local extrema, strict local extrema, global extrema and strict global extrema of $f$ are global subanalytic subsets of $\R^n$ and their images under $f$ are finite sets.

However, {\em we point out that further work is required for general weak categories like those of subanalytic, semianalytic or $C$-semianalytic sets}.

\subsection{Structure of the article}
The article in organized as follows. In Section \ref{s2} we present some preliminaries concerning local extrema of real functions, basic properties about subanalytic, semianalytic and $C$-semianalytic sets and functions and some enlightening examples. In Section \ref{s3} we prove the main results of this article and we analyze the properties of the set of openness points of a subanalytic function. In Section \ref{s4} we study local extrema of real analytic functions on real manifolds that are locally normal crossings (Theorem \ref{easydesc}) and we take advantage of local uniformization of a continuous subanalytic function $f:X\to\R$ on a closed subanalytic subset $X$ of a real analytic manifold $M$ to provide an alternative description as subanalytic subsets of $M$ of the sets ${\rm Max}_\lambda(f)$ for $\lambda\in\R$, that does not involve closures (Corollary \ref{altdesc}). 

\section{Basic facts, tools and examples}\label{s2}

In this section we present some preliminaries we need along the article in orden to lighten the proofs of the involved results. We begin analyzing some basic properties of local extrema of real functions. 

\subsection{Local extrema of real functions}
Let us see next some basic results concerning local extrema of real functions. We use here the letters $X,Y,Z$ to denote topological spaces.

\begin{lem}[(Alternative description of local maxima)]\label{altdes}
Let $f:X\to\R$ be a function and let $\lambda\in\R$. Then ${\rm Max}_\lambda(f)=\{f-\lambda=0\}\setminus\cl(\{f-\lambda>0\})$.
\end{lem}
\begin{proof}
Pick $x_0\in{\rm Max}_\lambda(f)$. Then there exists an open neighborhood $V^{x_0}\subset X$ of $x_0$ such that $f(x)\leq f(x_0)=\lambda$ for each $x\in V^{x_0}$, that is, $V^{x_0}\subset\{f-\lambda\leq0\}$, so $V^{x_0}\cap\{f-\lambda>0\}=\varnothing$. Thus, $x_0\in\{f-\lambda=0\}\setminus\cl(\{f-\lambda>0\})$. The converse inclusion is clear.
%Conversely, if $x_0\in\{f-\lambda=0\}\setminus\cl(\{f-\lambda>0\})$, then $f(x_0)=\lambda$ and there exists an open neighborhood $V^{x_0}\subset X$ of $x_0$ such that $V^{x_0}\cap\{f-\lambda>0\}=\varnothing$, that is, $V^{x_0}\subset\{f\leq\lambda\}$. Consequently, $x_0\in{\rm Max}_\lambda(f)$.
\end{proof}

\begin{remarks}\label{imext}\em
Let $f:X\to\R$ be a real function. 

(i) Let $\varphi:Y\to X$ be continuous map. Let $x_0\in X$ be a local maximum of $f$ and let $y_0\in Y$ be such that $\varphi(y_0)=x_0$. As $\varphi$ is continuous, $y_0$ is a local maximum of $f\circ\varphi$.

(ii) If $x_0\in{\rm Max}(f)\cap{\rm Min}(f)$, then there exists an open neighborhood $V^{x_0}\subset X$ such that $f(y)\leq f(x_0)\leq f(y)$ for each $y\in V^{x_0}$, that is, $f|_{V^y}$ is constant.
\end{remarks}

\begin{lem}\label{equiv}
Let $f:X\to\R$ be a function and let $x_0\in X$. Let $\pi:Y\to X$ be a surjective continuous closed map. Then $x_0$ is a local maximum of $f$ if and only if each point $y\in\pi^{-1}(x_0)$ is a local maximum of $f\circ\pi$. 
\end{lem}
\begin{proof}
The only if part follows from the continuity of $f$ (Remark \ref{imext}). To prove the if part we proceed as follows. Assume $f(x_0)=0$, so $(f\circ\pi)(y)=0$ for each $y\in\pi^{-1}(x_0)$. Thus, for each $y\in\pi^{-1}(x_0)$ there exists an open neighborhood $V^y\subset Y$ such that $(f\circ\pi)(z)\leq 0$ for each $z\in V^y$. Consequently, $(f\circ\pi)(z)\leq 0$ for each $z\in V:=\bigcup_{y\in\pi^{-1}(x_0)}V^y$. As $V$ is an open neighborhood of $\pi^{-1}(x_0)$ in $Y$, the difference $T:=Y\setminus V$ is a closed subset of $Y$. As $\pi$ is closed, $\pi(T)$ is a closed subset of $X$ that does not contain $x_0$. Thus, $U:=X\setminus\pi(T)$ is an open neighborhood of $x_0$ such that $\pi^{-1}(U)=Y\setminus\pi^{-1}(\pi(T))\subset Y\setminus T=V$, so $f(U)=(f\circ\pi)(\pi^{-1}(U))\subset (f\circ\pi)(V)\subset(-\infty,0]$, whereas $f(x_0)=0$. Consequently, $x_0$ is a local maximum of $f$, as required.
\end{proof}

\begin{lem}[(Reduction to the closed case)]\label{reduction}
Let $f:X\to Z$ be a map on a subset $X$ of a locally compact Hausdorff topological space $Y$ and let $\lambda\in\R$. Let $\Gamma\subset X\times Z\subset Y\times Z$ be the graph of $f$ and let $\ol{\Gamma}$ be the closure of $\Gamma$ in $Y\times Z$. Let $\pi_1:Y\times Z\to Y$ and $\pi_2:Y\times Z\to Z$ be the projections onto the first and the second factors of $Y\times Z$ and let $\rho:=\pi_2|_{\ol{\Gamma}}:\ol{\Gamma}\to Z$ be the restriction of $\pi_2$ to $\ol{\Gamma}$. We have:
\begin{itemize}
\item[(i)] $\pi_1(\ol{\Gamma}\setminus\Gamma)=\cl(X)\setminus X$ and $\pi_1(\Gamma)=X$.
\item[(ii)] The map $\rho:=\pi_2|_{\ol{\Gamma}}:\ol{\Gamma}\to Z$ is continuous.
\item[(iii)] If $f$ is $Y$-compact, the restriction $\pi_1|_{\ol{\Gamma}}:\ol{\Gamma}\to Y$ is a proper map. 
\end{itemize}
Assume in what follows $Z=\R$. We have:
\begin{itemize}
\item[(iv)] ${\rm Max}_\lambda(f)=\pi_1({\rm Max}_\lambda(\rho))\cap X=\pi_1({\rm Max}_\lambda(\rho)\cap\Gamma)$ for each $\lambda\in\R$.
\item[(v)] If $f$ is $Y$-compact and the family $\{{\rm Max}_\lambda(\rho)\}_{\lambda\in\R}$ is locally finite in $Y\times\R$, the family $\{{\rm Max}_\lambda(f)\}_{\lambda\in\R}$ is locally finite in $Y$.
\item[(vi)] ${\rm Max}(f)=\pi_1({\rm Max}(\rho))\cap X=\pi_1({\rm Max}(\rho)\cap\Gamma)$.
\end{itemize}
\end{lem}
\begin{proof}
Statements (i) and (ii) are straightforwardly proved.

(iii) Let $K\subset Y$ be a compact set and let us check: \em $\pi_1^{-1}(K)\cap\ol{\Gamma}$ is a compact subset of $\ol{\Gamma}$\em. This will prove that $\pi_1|_{\ol{\Gamma}}$ is proper.

Let $K'$ be another compact subset of $Y$ that contains $K$ in its interior. Then the closure $C:=\cl(f(K'\cap X))$ is a compact subset of $Z$. Pick a point $(y,z)\in\pi_1^{-1}(K)\cap\ol{\Gamma}$ and let $V\times W\subset Y\times Z$ be an open neighborhood of $(y,z)$. Then $y=\pi_1(y,z)\in K$ and we may assume that $V\subset K'$. As $(y,z)\in\ol{\Gamma}$, there exists $x\in X$ such that $(x,f(x))\in\Gamma\cap(V\times W)$. Observe that $(x,f(x))\in K'\times C$. Thus, each open neighborhood of $(y,z)$ meets the closed set $K'\times C$, so $(y,z)\in K'\times C$. Consequently, $\pi_1^{-1}(K)\cap\ol{\Gamma}$ is a closed subset of $Y\times Z$ contained in the compact subset $K'\times C$. Hence, $\pi_1^{-1}(K)\cap\ol{\Gamma}$ is a compact subset of $Y\times Z$.

(iv) We prove next: ${\rm Max}_\lambda(f)=\pi_1({\rm Max}_\lambda(\rho))\cap X$. 

Pick $x_0\in{\rm Max}_\lambda(f)$ and let us prove $x_0\in\pi_1({\rm Max}_\lambda(\rho))\cap X$. There exists an open neighborhood $V^{x_0}\subset Y$ of $x_0$ such that $f(x_0)\geq f(x)$ for each $x\in X\cap V^{x_0}$. We distinguish two cases:

\noindent{\sc Case 1}: If $x_0\in X\setminus\cl(\cl(X)\setminus X)$, we may assume $\cl(\cl(X)\setminus X)\cap V^{x_0}=\varnothing$. Thus, 
$$
\ol{\Gamma}\cap(V^{x_0}\times\R)=\Gamma\cap(V^{x_0}\times\R)=\{(x,f(x)):\ x\in X\cap V^{x_0}\}.
$$
If $(y,t)\in\ol{\Gamma}\cap(V^{x_0}\times\R)$, we have $\rho(y,t)=\pi_2(y,f(y))=f(y)\leq f(x_0)=\rho(x_0,f(x_0))$, so the point $(x_0,f(x_0))\in{\rm Max}_\lambda(\rho)$ and $x_0\in\pi_1({\rm Max}_\lambda(\rho))\cap X$. 

\noindent{\sc Case 2}: Assume next that $x_0\in\cl(\cl(X)\setminus X)$ and pick $(y,t)\in\ol{\Gamma}\cap(V^{x_0}\times\R)$, then $y\in\cl(X)\cap V^{x_0}$. Let $(x,f(x))\in\Gamma$ be close to $(y,t)$. Then $x\in X\cap V^{x_0}$, so $f(x)\leq f(x_0)$. Consequently, $t\leq f(x_0)$ and $\rho(y,t)=\pi_2(y,t)=t\leq f(x_0)=\rho(x_0,f(x_0))$, so $(x_0,f(x_0))\in{\rm Max}_\lambda(\rho)$ and $x_0\in\pi_1({\rm Max}_\lambda(\rho))\cap X$.

Conversely, let $x_0\in\pi_1({\rm Max}_\lambda(\rho))\cap X$. Then there exists $(x_0,t_0)\in{\rm Max}_\lambda(\rho)$. As $\pi_1(x_0,t_0)\in X$, then $(x_0,t_0)\in\Gamma$ and $t_0=f(x_0)$. As $(x_0,f(x_0))\in{\rm Max}_\lambda(\rho)$, there exist an open neighborhood $W^{x_0}\subset Y$ of $x_0$ and $\veps>0$ such that if $(x,t)\in\ol{\Gamma}\cap(W^{x_0}\times(f(x_0)-\veps,f(x_0)+\veps))$, then $t=\rho(x,t)\leq\rho(x_0,f(x_0))=f(x_0)$. In addition, if $(x,t)\in\Gamma\cap(W^{x_0}\times(f(x_0)-\veps,f(x_0)+\veps))$, it holds $f(x)=t\leq f(x_0)$. Thus, $f(x)=t\leq f(x_0)$ for each $x\in X\cap W^{x_0}$, so $x_0\in{\rm Max}_\lambda(f)$.

By (i) we have $\pi_1({\rm Max}_\lambda(\rho))\cap X=\pi_1({\rm Max}_\lambda(\rho)\cap\Gamma)$.

(v) Let $x_0\in X$ and let $\xi_0:=(x_0,f(x_0))\in\ol{\Gamma}$ be the unique point in $\ol{\Gamma}$ such that $\pi_1(\xi_0)=x_0$. As the family $\{{\rm Max}_\lambda(\rho)\}_{\lambda\in\R}$ is locally finite, there exists an open neighborhood $V^{\xi_0}\subset\ol{\Gamma}$ such that $V^{\xi_0}$ meets only finitely many ${\rm Max}_\lambda(\rho)$. As $f$ is $Y$-compact, the restriction $\pi_1|_{\ol{\Gamma}}:\ol{\Gamma}\to Y$ is a proper map. Then $C:=\pi_1(\ol{\Gamma}\setminus V^{\xi_0})$ is a closed subset of $Y$ that does not contain $x_0$. Thus, $U^{x_0}:=Y\setminus C$ is an open neighborhood of $x_0$ in $Y$. 

We have $\pi_1^{-1}(U^{x_0})\cap\ol{\Gamma}=\ol{\Gamma}\setminus\pi_1^{-1}(\pi_1(\ol{\Gamma}\setminus V^{\xi_0}))\subset V^{\xi_0}$. Let $\lambda\in\R$ be such that $U^{x_0}$ meets ${\rm Max}_\lambda(f)=\pi_1({\rm Max}_\lambda(\rho)\cap\Gamma)$. Then $\pi_1^{-1}(U^{x_0})\subset V^{\xi_0}$ meets 
$$
\pi_1^{-1}(\pi_1({\rm Max}_\lambda(\rho)\cap\Gamma))={\rm Max}_\lambda(\rho)\cap\Gamma
$$ 
(the last equality holds because $\pi_1|_{\Gamma}:\Gamma\to X$ is bijective and $\pi_1(\ol{\Gamma}\setminus\Gamma)=\cl(X)\setminus X$). Consequently, there are finitely many $\lambda\in\R$ such that $U^{x_0}$ meets $\pi_1({\rm Max}_\lambda(\rho)\cap\Gamma)={\rm Max}_\lambda(f)$. We conclude that the family $\{{\rm Max}_\lambda(f)\}_{\lambda\in\R}$ is locally finite in $Y$.

(vi) As ${\rm Max}(\rho)=\bigsqcup_{\lambda\in\R}{\rm Max}_\lambda(\rho)$ and ${\rm Max}(f)=\bigsqcup_{\lambda\in\R}{\rm Max}_\lambda(f)$, we have
\begin{multline*}
{\rm Max}(f)=\bigsqcup_{\lambda\in\R}{\rm Max}_\lambda(f)=\bigsqcup_{\lambda\in\R}\pi_1({\rm Max}_\lambda(\rho))\cap X\\
=\pi_1\Big(\bigsqcup_{\lambda\in\R}{\rm Max}_\lambda(\rho)\Big)\cap X=\pi_1({\rm Max}(\rho))\cap X.
\end{multline*}
In addition, the equality $\pi_1({\rm Max}(\rho))\cap X=\pi_1({\rm Max}(\rho)\cap\Gamma)$ follows from statement (i). 
\end{proof}

\begin{lem}[(Locally isolated description of local maxima)]\label{isolated}
Let $f:X\to\R$ be a continuous function on a topological space $X$ and assume that the family $\{{\rm Max}_\lambda(f)\}_{\lambda\in\R}$ is locally finite. We have:
\begin{itemize}
\item[(i)] If $x\in\cl({\rm Max}(f))$, there exists an open neighborhood $V^x$ that meets only ${\rm Max}_{f(x)}(f)$. In particular, the family $\{{\rm Max}_\lambda(f)\}_{\lambda\in\R}$ is locally isolated in $X$.
\item[(ii)] Assume in addition that the family $\{{\rm Min}_\lambda(f)\}_{\lambda\in\R}$ is locally finite. If $x_0\in X$ is a local maximum of $f$, there exists an open neighborhood $U\subset X$ of $x_0$ such that ${\rm Extr}(f)\cap U={\rm Max}(f)\cap U={\rm Max}_{f(x_0)}(f)\cap U=\{f-f(x_0)=0\}\cap U$.
\item[(iii)] For each $\lambda\in\R$ there exists and open neighborhood $V\subset X$ of $\{f=\lambda\}$ such that ${\rm Max}(f)\cap V={\rm Max}_\lambda(f)$.
\end{itemize}
\end{lem}
\begin{proof}
(i) Assume $x\in\cl({\rm Max}(f))=\bigsqcup_{\lambda\in\R}\cl({\rm Max}_\lambda(f))\subset\bigsqcup_{\lambda\in\R}\{f=\lambda\}$ and write $\mu:=f(x)$. Then $x$ belongs only to $\cl({\rm Max}_\mu(f))$. Let $U\subset X$ be an open neighborhood of $x$ that meets only finitely many of the sets ${\rm Max}_\lambda(f)$, say for the distinct values $\mu,\lambda_1,\ldots,\lambda_r\in\R$. Then each open neighborhood $V^x\subset U\setminus\bigcup_{i=1}^r\{f=\lambda_i\}$ of $x$ meets only ${\rm Max}_\mu(f)$.

(ii) Write $\mu:=f(x_0)$ and let $U_0\subset X$ be an open neighborhood of $x_0$ such that ${\rm Max}(f)\cap U_0={\rm Max}_\mu(f)\cap U_0$. As ${\rm Max}_\mu(f)=\{f=\mu\}\setminus\cl(\{f>\mu\})$, the difference $V:=U_0\setminus\cl(\{f>\mu\})\subset X$ is an open neighborhood of $x_0$ such that ${\rm Max}(f)\cap V=\{f=\mu\}\cap V$. By (i) and the previous argument (applied to $-f$) we find an open neighborhood $W\subset X$ of $x_0$ such that either ${\rm Min}(f)\cap W={\rm Min}_\mu(f)\cap W=\varnothing$ or ${\rm Min}(f)\cap W={\rm Min}_\mu(f)\cap W=\{f=\mu\}\cap W$. If we define $U:=V\cap W$, we have ${\rm Extr}(f)\cap U=\{f-f(x_0)=0\}\cap U={\rm Max}(f)\cap U={\rm Max}_{f(x_0)}(f)\cap U$, as required.

(iii) Fix $\lambda\in\R$ and pick $x\in Z_\lambda:=\{f=\lambda\}$. By (i) there exists an open neighborhood $V^x\subset X$ of $x$ such that ${\rm Max}(f)\cap V^x={\rm Max}_\lambda(f)\cap V^x$. Define $V:=\bigcup_{x\in X}V^x$ and observe that ${\rm Max}(f)\cap V={\rm Max}_\lambda(f)$, as required.
\end{proof}

\subsection{Basic properties of ${\mathfrak F}$-maps}
We study next some basic properties of ${\mathfrak F}$-maps where ${\mathfrak F}$ is a weak category. We will focus mainly in the cases when ${\mathfrak F}$ is either the subanalytic, semianalytic or $C$-semianalytic categories. Contrary to what happens in the semialgebraic case, the composition of ${\mathfrak F}$-maps needs not to be an ${\mathfrak F}$-map. Consider for instance the subanalytic subset $X:=\R\setminus\{0\}$ of $\R$ and the subanalytic functions $f:X\to\R,\ x\mapsto 1/x$ and $g:\R\to\R,\ y\mapsto\sin(y)$. The composition $g\circ f:X\to\R,\ x\mapsto\sin(\frac{1}{x})$ is not a subanalytic function. 

Let $M,N,P$ denote real analytic manifolds. We recall next for the sake of completeness well-known sufficient conditions to guarantee that the composition of two subanalytic functions is subanalytic.

\begin{lem}\label{compos}
Let $f:X\subset M\to N$ and $g:Y\subset N\to P$ be subanalytic maps such that $f(X)\subset Y$. Assume that either $f$ is $M$-compact or the inverse image under $g$ of each compact subset of $P$ under $g$ is relatively compact in $N$. Then $g\circ f$ is subanalytic.
\end{lem}
\begin{proof}
Let $\Gamma_f:=\{(x,f(x)):\ x\in X\}$ and $\Gamma_g:=\{(y,g(y)):\ y\in Y\}$ be the graphs of $f$ and $g$. The set 
$$
T:=(\Gamma_f\times\Gamma_g)\cap\{(x,u,y,v)\in M\times N\times N\times P:\ u=y\}
$$
is a subanalytic subset of $M\times N\times N\times P$. Thus, $Z:=\{(x,f(x),g(f(x))):\ x\in X\}$ is a subanalytic subset of $M\times N\times P$. Consider the projection $\pi:M\times N\times P\to M\times P,\ (x,u,v)\mapsto (x,v)$. If we prove that the restriction $\pi|_{\cl(Z)}$ is proper, $\Gamma_{g\circ f}=\pi(Z)$ is a subanalytic subset of $M\times P$ and $g\circ f$ is subanalytic. Let $K\subset M\times P$ be a compact set and let us check that $\pi^{-1}(K)\cap\cl(Z)$ is compact. Let $K_1$ be a compact subset of $M$ and let $K_2$ be a compact subset of $P$ such that $K\subset\Int(K_1)\times\Int(K_2)$. One can check:
\begin{align*}
&\pi^{-1}(K)\cap\cl(Z)\subset K_1\times\cl(f(K_1\cap X))\times K_2,\\
&\pi^{-1}(K)\cap\cl(Z)\subset K_1\times\cl(g^{-1}(K_2))\times K_2.
\end{align*}
These inclusions imply under the hypothesis of the statement that $\pi^{-1}(K)\cap\cl(Z)$ is compact, as required.
\end{proof}
\begin{remark}\em
If either $f$ is continuous and $X$ is closed in $M$ or $g$ is proper, then $g\circ f$ is subanalytic (because the hypotheses of Lemma \ref{compos} are fulfilled).
\end{remark}

Similarly, inverse images of ${\mathfrak F}$-sets under ${\mathfrak F}$-maps need not to be ${\mathfrak F}$-sets (see Example \ref{ptd}). If ${\mathfrak F}$ is the category of subanalytic sets, a sufficient condition to guarantee that the inverse image under a subanalytic function $f:X\subset M\to N$ of a subanalytic subset of $N$ is a subanalytic subset of $M$ is that $f$ is $M$-compact. Let us recall next how this is easily proved. 

\begin{lem}\label{pre}
Let $X\subset M$ be a subanalytic set and let $f:X\to N$ be an $M$-compact subanalytic map. Then $f^{-1}(Y)$ is a subanalytic subset of $M$ for each subanalytic subset $Y$ of $N$.
\end{lem}
\begin{proof}
Denote $\Gamma_f$ the graph of $f$ and consider the projection $\rho:M\times N\to M$ onto the first factor. As $f$ is subanalytic, $\cl(\Gamma_f)$ is a subanalytic subset of $M\times N$. Observe that $f^{-1}(Y)=\rho(\Gamma_f\cap(M\times Y))$ and $\Gamma_f\cap(M\times Y)$ is a subanalytic subset of $M\times N$. As the restriction $\rho|_{\cl(\Gamma_f)}:\Gamma_f\to M,\ (x,y)\mapsto x$ is by Lemma \ref{reduction} proper, $f^{-1}(Y)$ is a subanalytic subset of $M$.
\end{proof}

If $X\subset M$ is a closed subanalytic set and $f:X\to N$ is a continuous subanalytic function, $f$ is an $M$-compact subanalytic map. Thus, $f^{-1}(Y)$ is a subanalytic subset of $M$ for each subanalytic subset $Y$ of $N$. In the semianalytic and the $C$-semianalytic categories, we can guarantee that the inverse image of an ${\mathfrak F}$-set under an ${\mathfrak F}$-map is an ${\mathfrak F}$-set if we limit our scope to the restrictions to elements of ${\mathfrak F}(M)$ of analytic maps $f:M\to N$. However, the behavior with respect to the fibers is always net.

\begin{lem}
Let ${\mathfrak F}$ denote either the subanalytic, semianalytic or $C$-semianalytic categories. Let $X\subset M$ and let $f:X\to N$ be a ${\mathfrak F}$-map. Then the fibers of $f$ are ${\mathfrak F}$-sets.
\end{lem}
\begin{proof}
Assume first that ${\mathfrak F}$ is the category of subanalytic sets and let $p\in N$. Let $\rho:M\times N\to M$ be the projection onto the first factor, which is an analytic function. It holds that $f^{-1}(p)=\rho(\Gamma_f\cap(M\times\{p\}))$. If we prove that the restriction 
$$
\rho|_{\cl(\Gamma_f\cap(M\times\{p\}))}:\cl(\Gamma_f\cap(M\times\{p\}))\to M
$$ 
is proper, $f^{-1}(p)$ is a subanalytic subset of $M$.

We have $\cl(\Gamma_f\cap(M\times\{p\}))\subset M\times\{p\}$. Let $K$ be a compact subset of $M$. Then $\rho^{-1}(K)\cap\cl(\Gamma_f\cap(M\times\{p\}))\subset K\times\{p\}$ is a closed subset of a compact set, so it is a compact subset of $M\times N$. Consequently, $\rho|_{\cl(\Gamma_f\cap(M\times\{p\}))}$ is a proper map.

Assume next ${\mathfrak F}$ is the category of semianalytic sets. Let $p\in N$ and let $(x,y)\in M\times N$. As the graph $\Gamma_f$ is a semianalytic set, there exist open neighborhoods $V$ of $x$ and $W$ of $p$ and finitely many analytic functions $g_i,f_{ij}\in\an(V\times W)$ such that 
$$
\Gamma_f\cap(V\times W)=\bigcup_i\{f_{i1}>0,\ldots,f_{is}>0,g_i=0\}.
$$
Thus, $f^{-1}(p)\cap V=\bigcup_i\{f_{i1}(\x,p)>0,\ldots,f_{is}(\x,p)>0,g_i(\x,p)=0\}$, where $g_i(\x,p),f_{ij}(\x,p)\in\an(V)$. Consequently, $f^{-1}(p)$ is a semianalytic subset of $M$.

If ${\mathfrak F}$ is the category of $C$-semianalytic sets, the proof is similar.
\end{proof}

\begin{remark}\em
Let ${\mathfrak F}$ be a weak category of subanalytic sets. There are continuous maps between real algebraic manifolds such that the inverse image of each ${\mathfrak F}$-set is an ${\mathfrak F}$-set, but such maps are not even subanalytic. Let $f:\R\to\R$ be any strictly increasing continuous function whose graph is not a subanalytic subset of $\R\times\R$, so $f$ provides a homeomorphism between $\R$ and an open interval of $\R$. Observe that the inverse image of an interval is again an interval. If $S$ is an ${\mathfrak F}$-subset of $\R$, it is a locally finite union of points and intervals. Thus, the same happens with its pre-image under $f$. Consequently, $f$ is a function such that the inverse image of each ${\mathfrak F}$-subset of $\R$ is an ${\mathfrak F}$-subset of $\R$ but $f$ is not itself subanalytic. In fact, the inverse image of each semialgebraic subset of $\R$ under $f$ is a semialgebraic subset of $\R$ and $f$ is not semialgebraic.
\end{remark}

We present next some examples to enlighten some particularities of ${\mathfrak F}$-maps that differ from the net behavior of semialgebraic maps or more generally definable maps of an o-minimal structure.

\begin{examples}\label{ptd}\em
(i) For each integer $m\geq1$ consider the $C$-semianalytic set $X_m:=\{m\x\geq\y>(m-1)\x>0\}$. Observe that $X:=\bigsqcup_{m\geq1}X_m=\{\x>0,\y>0\}$ is a $C$-semianalytic subset of $M:=\R^2$. Consider the function $f:X\to\R$ such that $f|_{X_m}=(-1)^mm$. The graph 
$$
\Gamma_f=\bigsqcup_{m\geq1}X_m\times\{(-1)^mm\}
$$
is a $C$-semianalytic subset of $\R^3$, because it is a locally finite union of basic $C$-semianalytic sets. Thus, $f$ is a $C$-semianalytic function, but it is not $M$-compact, because $f((0,1]^2)=\Z$. Observe that $Y:=f^{-1}((0,+\infty))=\bigsqcup_{k\geq1}X_{2k}$ is not a subanalytic subset of $\R^2$ because its family of connected components is not locally finite in $\R^2$. However, the restriction $g:=f|_Y:Y\to\R$ is still a $C$-semianalytic function, because its graph $\Gamma_g=\bigsqcup_{k\geq1}X_{2k}\times\{2k\}$ is a $C$-semianalytic subset of $\R^3$. The sets ${\rm Min}(f)=\bigsqcup_{k\geq1}\Int(X_{2k})\sqcup\bigsqcup_{k\geq1}X_{2k-1}$ and ${\rm Max}(f)=\bigsqcup_{k\geq1}X_{2k}\sqcup\bigsqcup_{k\geq1}\Int(X_{2k-1})$ are not subanalytic subsets of $\R^2$, again because their families of connected components are not locally finite in $\R^2$.

Consider also the continuous function 
$$
g:X=\bigsqcup_{k\geq1}(X_{2k}\cup X_{2k-1})\to\R,\ (x,y)\mapsto
\begin{cases}
\frac{y}{x}-k&\text{if $(x,y)\in X_{2k}$,}\\
k-1&\text{if $(x,y)\in X_{2k-1}$,}
\end{cases}
$$
whose graph is $\Gamma:=\bigsqcup_{k\geq1}(\Gamma_{2k}\cup\Gamma_{2k-1})$ where
$$
\Gamma_{2k-1}:=X_{2k-1}\times\{k-1\}\quad\text{and}\quad\Gamma_{2k}:=\{2k\x\geq\y>(2k-1)\x>0,\x\z=\y-k\x\}.
$$
As the family $\{\Gamma_{2k-1},\Gamma_{2k}\}_{k\geq1}$ is locally finite in $\R^3$, the graph of $g$ is a $C$-semianalytic subset of $\R^3$. Thus, $g$ is $C$-semianalytic, but it is not $M$-compact because it maps $X\cap[0,1]^2$ to $[0,+\infty)$. The sets ${\rm Min}(g)=\bigsqcup_{k\geq1}X_{2k-1}$ and ${\rm Max}(g)=\bigsqcup_{k\geq1}Y_{2k-1}$, where 
$$
Y_{2k-1}:=(\cl(X_{2k-1})\cap X\setminus X_{2k-1})\cup\Int(X_{2k-1}),
$$
are not subanalytic because the families of their connected components are not locally finite in $\R^2$.

Let $Y:=X\cap\{\y\leq1\}$ and consider the subanalytic function 
$$
h:Y\to\R,\ (x,y)\mapsto\begin{cases}
g(x,y)&\text{if $(x,y)\in X\cap\{\y<1\}$,}\\
k&\text{if $(x,y)\in(X_{2k+1}\cup X_{2k})\cap\{\y=1\}$ for $k\geq1$.}
\end{cases}
$$
The graph of $h$ is the subanalytic subset of $\R^3$ given by
$$
(\Gamma\cap((X\cap\{\y<1\})\times\R))\cup\bigcup_{k\geq1}((X_{2k+1}\cup X_{2k})\cap\{\y=1\}\times\{k\}),
$$
so $h$ is subanalytic (although it is not continuous). Observe that ${\rm Max}(h)=\{\x>0,\y=1\}\cup\bigsqcup_{k\geq1}Y_{2k-1}\cap\{\y<1\}$ is a connected set. However, ${\rm Max}(h)$ it is not a subanalytic subset of $\R^2$. Otherwise ${\rm Max}(h)\cap\{\y<1\}=\bigsqcup_{k\geq1}(Y_{2k-1}\cap\{\y<1\})$ will be subanalytic, but it is not because the family of its connected components is not locally finite in $\R^2$. Thus, ${\rm Max}(h)$ is not subanalytic although it is connected (contrast this with Theorem \ref{main0} where the continuity of the subanalytic function is assumed).

(ii) Consider the analytic map $g:\R^2\to\R^3,\ (x_1,x_2)\mapsto(x_1,x_1x_2,x_1e^{x_2})$ and the subanalytic subset $X:=g([-1,1]^2\setminus\{\x_1=0\})=g([-1,1]^2)\setminus\{(0,0,0)\}$ of $\R^3$ (Osgood's example), which is not a semianalytic subset of $\R^2$. Define the continuous function 
$$
f:X\to\R,\ (x_1,x_2,x_3)\mapsto
\frac{x_2}{x_1}
$$
and note that
$$
f\circ g:[-1,1]^2\setminus\{\x_1=0\}\to\R,\ (x_1,x_2)\mapsto x_2.
$$
The graph of $f$ is 
\begin{multline*}
\Gamma_f:=\{(x_1,x_1x_2,x_1e^{x_2},x_2):\ (x_1,x_2)\in[-1,1]^2\setminus\{\x_1=0\}\},\\
=\{\x_2=\x_1\x_4,\x_3=\x_1e^{\x_4},0<|\x_1|\leq 1,|\x_2|\leq1,\},
\end{multline*}
which is a $C$-semianalytic subset of $\R^4$. Thus $f$ is a $C$-semianalytic function, whereas its domain $X$ is not even a semianalytic set. Observe that $f(X)=(f\circ g)([-1,1]^2\setminus\{\x_1=0\})=[-1,1]$. In addition, \em the inverse image under $f$ of an interval $I$ of $\R$ is a semianalytic subset of $\R^3$ if and only if $I\cap[-1,1]$ is a singleton\em.

If $\lambda\in\R$, then the fiber $f^{-1}(\lambda)$ is either the empty set if $\lambda\not\in[-1,1]$ or the $C$-semianalytic set $\{x_1\cdot(1,\lambda,e^\lambda):\ 0<|x_1|\leq1\}$ otherwise. Let $I\subset\R$ be an interval and denote $J:=I\cap[-1,1]$. We have $f^{-1}(I)=f^{-1}(J)=f^{-1}(\Int(J))\sqcup f^{-1}(\partial J)$. As $\partial J$ is a finite set, $f^{-1}(\partial J)$ is a $C$-semianalytic subset of $\R^3$. Observe that $\Int(J):=(a,b)\subset[-1,1]$ where $a\leq b$ and 
\begin{multline*}
f^{-1}((a,b))=g((f\circ g)^{-1}(a,b))\\
=g(([-1,1]\setminus\{0\})\times(a,b))=\{(x_1,x_1x_2,x_1e^{x_2}):\ 0<|x_1|\leq1,\ a<x_2<b\},
\end{multline*}
which is a subanalytic subset of $\R^3$ of dimension $2$ if $a<b$ because it is the image of a semianalytic set under the proper analytic map $g|_{[-1,1]^2}$. Let us check that $f^{-1}((a,b))$ is not a semianalytic subset of $\R^3$ if $a<b$. We claim: \em If $G(u,v,w)$ is an analytic function in three variables on a neighborhood of the origin such that $G(x_1,x_1x_2,x_1e^{x_2})=0$ for $(x_1,x_2)\in([-1,1]\setminus\{0\})\times(a,b)$ with $x$ close to $0$, then $G=0$\em.

Write $G(u,v,w)=\sum_{j\geq0}G_j(u,v,w)$ where $G_j(u,v,w)$ is a homogeneous polynomial of degree $j$. Then for each fixed $y\in(a,b)$ we have
$$
0=G(x_1,x_1x_2,x_1e^{x_2})=\sum_{j\geq0}G_j(x_1,x_1x_2,x_1e^{x_2})=\sum_{j\geq0}x_1^jG_j(1,x_2,e^{x_2})
$$
for $0<|x_1|\leq1$ close to $0$. Therefore, $G_j(1,x_2,e^{x_2})=0$ for each $j\geq0$ and each $y\in(a,b)$, so each $G_j=0$ and $G=0$. Consequently, the smallest real analytic set containing (the germ at the origin of) $f^{-1}((a,b))$ is the whole $\R^3$, so the $2$-dimensional subanalytic set $f^{-1}((a,b))$ is not semianalytic.

(iii) Let $\lambda:\N\to\Q$ be a bijection such that $\lambda(0)=0$ and denote $\lambda_n:=\lambda(n)$. Consider the sequence
$$
\mu_n:=\begin{cases}
\lambda_k&\text{if $n=2k$},\\
\min\{\lambda_k,\lambda_{k+1}\}-1&\text{if $n=2k+1$}
\end{cases}
$$
for each $n\geq0$. Let $f:[0,+\infty)\to\R$ be the continuous subanalytic function whose graph is the polygonal that connects orderly the points $(n,\mu_n)$ for $n\geq0$ and let us extend it (continuously and subanalytically) to $\R$ defining $f(x)=-x$ for each $x<0$. Observe that ${\rm Max}(f)=\{2k:\ k\in\N\}$ and $f({\rm Max}(f))=\Q$, which is a dense countable subset of $\R$. Contrast this example with Lemma \ref{semiom}.
\end{examples}

\section{Proof of the main results}\label{s3}

In this section we prove the main results of this work (Theorem \ref{main0}, Theorem \ref{main1}, Lemma \ref{main2} and Proposition \ref{sharp}). We begin by showing Lemma \ref{main2}, that is, if $f:X\to\R$ is a continuous subanalytic function on a closed subanalytic subset $X$ of a real analytic manifold $M$, then the family $\{{\rm Max}_\lambda(f)\}_{\lambda\in\R}$ is locally finite in $M$ and each set ${\rm Max}_\lambda(f)$ and ${\rm Max}(f)$ are subanalytic subsets of $M$. 

\begin{proof}[of Lemma \em\ref{main2}]
Let $\phi:M\hookrightarrow\R^p$ be an analytic immersion of $M$ as a closed analytic submanifold of $\R^p$ and let us identify $M$ with $\phi(M)$. Thus, $X$ is a closed subanalytic subset of $\R^p$ and $f:X\to\R$ is an $M$-compact function. Consider the analytic diffeomorphism $\psi:\R\to(-1,1),\ t\mapsto\frac{t}{\sqrt{1+t^2}}$. As $f$ is $M$-compact, the composition $\psi\circ f:X\to(-1,1)$ is a continuous subanalytic function, so we may assume in what follows $|f|<1$. Denote $\Gamma\subset\R^{p+1}$ the graph of $f$, which is a subanalytic subset of $\R^{p+1}$. 

Let $m\geq1$ be an integer and consider the hypercube $[-m,m]^n$ for each $n\geq1$. Consider the homothety $h_m:\R^p\to\R^p$ of center the origin and ratio $\frac{1}{m}$, which maps the hypercube $[-m,m]^p$ onto the hypercube $[-1,1]^p$ and the analytic diffeomorphism 
$$
\varphi:\R^p\to(-1,1)^p,\ x:=(x_1,\ldots,x_p)\mapsto\Big(\frac{x_1}{\sqrt{1+x_1^2}},\ldots,\frac{x_p}{\sqrt{1+x_p^2}}\Big).
$$
Let ${\mathfrak S}:=\{{\mathfrak S}_n\}_{n\geq1}$ be the o-minimal structure $\R_{\rm an}$. Observe that 
\begin{align*}
X_m&:=\varphi^{-1}(h_m(X\cap[-m,m]^p))\in{\mathfrak S}_p\\
\Gamma_m&:=(\varphi\times\id_\R)^{-1}((h_m\times\id_\R)(\Gamma\cap[-m,m]^{p+1}))\in{\mathfrak S}_{p+1}, 
\end{align*}
that is, $X_m$ is a definable set and $f_m:=f\circ h_m^{-1}\circ\varphi:X_m\to(-1,1)$ is a definable function of the o-minimal structure $\R_{\rm an}$. By Lemma \ref{semiom} the set ${\rm Max}(f_m)$ is a global subanalytic subset of $\R^p$ and the family $\{{\rm Max}_\lambda(f_m)\}_{\lambda\in\R}$ is finite. Thus, ${\rm Max}(f|_{X\cap(-m,m)^p})={\rm Max}(f)\cap(-m,m)^p$ is a subanalytic subset of $(-m,m)^p$ and the family $\{{\rm Max}_\lambda(f|_{X\cap(-m,m)^p})={\rm Max}_\lambda(f)\cap(-m,m)^p\}_{\lambda\in\R}$ is finite for each $m\geq1$. Consequently, the family $\{{\rm Max}_\lambda(f)\}_{\lambda\in\R}$ is locally finite in $\R^p$.

As $f$ is an $M$-compact subanalytic map, the inverse image under $f$ of the intervals of $\R$ are subanalytic subsets of $M$. Consequently, each set ${\rm Max}_\lambda(f)=\{f-\lambda=0\}\setminus\cl(\{f-\lambda>0\})$ is a subanalytic subset of $M$. Hence, ${\rm Max}(f)=\bigsqcup_{\lambda\in\R}{\rm Max}_\lambda(f)$ is a locally finite union of subanalytic subsets of $M$, so ${\rm Max}(f)$ is itself a subanalytic subset of $M$, as required. 
\end{proof}

\begin{proof}[of Theorem \em\ref{main1}]
As $f$ is an $M$-compact subanalytic function, the inverse image of each interval of $\R$ under $f$ is a subanalytic subset of $M$ (Lemma \ref{pre}). Thus, each set ${\rm Max}_\lambda(f)=\{f-\lambda=0\}\setminus\cl(\{f-\lambda>0\})$ is a subanalytic subset of $M$. Let $\Gamma_f\subset X\times\R\subset M\times\R$ be the graph of $f$ and let $\ol{\Gamma}_f$ be its closure in $M\times\R$. Let $\pi_1:M\times\R\to M$ and $\pi_2:M\times\R\to\R$ be the projections onto the first and the second factors of $M\times\R$ and let $\rho:=\pi_2|_{\ol{\Gamma}_f}:\ol{\Gamma}_f\to\R$ be the restriction of $\pi_2$ to the closed subanalytic subset $\ol{\Gamma}_f$ of $M\times\R$. As $\rho:\ol{\Gamma}_f\to\R$ is an analytic function on a closed subanalytic subset of $M\times\R$, we deduce by Lemma \ref{main2} that the family $\{{\rm Max}_\lambda(\rho)\}_{\lambda\in\R}$ is locally finite in $M\times\R$. As $f$ is an $M$-compact subanalytic function, we deduce by Lemma \ref{reduction} that the family $\{{\rm Max}_\lambda(f)\}_{\lambda\in\R}$ is locally finite in $M$. Thus, ${\rm Max}(f)=\bigsqcup_{\lambda\in\R}{\rm Max}_\lambda(f)$ is a subanalytic subset of $M$, as required.
\end{proof}

We are now ready to prove Theorem \ref{main0}.

\begin{proof}[of Theorem \em \ref{main0}]
The implications (i) $\Longrightarrow$ (ii) and (iv) $\Longrightarrow$ (i) are immediate (because ${\mathfrak F}$ is a weak category of subanalytic sets, the family of the connected components of a subanalytic subset of $M$ is a locally finite family of subanalytic subsets of $M$ and ${\mathfrak F}(M)$ is closed for locally finite unions). The implication (iii) $\Longrightarrow$ (iv) follows from Lemma \ref{isolated}. Let us prove next the remaining implication (ii) $\Longrightarrow$ (iii).

Let $\Gamma_f\subset X\times\R\subset M\times\R$ be the graph of $f$ and let $\ol{\Gamma}_f$ be its closure in $M\times\R$. Let $\pi_1:M\times\R\to M$ and $\pi_2:M\times\R\to\R$ be the projections onto the first and the second factors of $M\times\R$ and let $\rho:=\pi_2|_{\ol{\Gamma}_f}:\ol{\Gamma}_f\to\R$ be the restriction of $\pi_2$ to the closed subanalytic subset $\ol{\Gamma}_f$ of $M\times\R$. By Lemma \ref{reduction} we have ${\rm Max}_\lambda(f)=\pi_1({\rm Max}_\lambda(\rho)\cap\Gamma)$ for each $\lambda\in\R$. By Lemma \ref{main2} the family $\{{\rm Max}_\lambda(\rho)\}_{\lambda\in\R}$ is locally finite in $M\times\R$, so in particular it is a countable family. Thus, the family $\{{\rm Max}_\lambda(f)\}_{\lambda\in\R}$ is countable and $f({\rm Max}(f))=\bigcup_{\lambda\in\R}f({\rm Max}_\lambda(f))$ is a countable subset of $\R$.

Let $C$ be a connected component of ${\rm Max}(f)$. As $f$ is continuous, $f(C)\subset f({\rm Max}_\lambda(f))$ is connected, so it is a singleton. If we write $f(C)=\{\lambda\}$, we have $C\subset{\rm Max}(f)\cap\{f=\lambda\}={\rm Max}_\lambda(f)$.

Let $\{C_j\}_{j\in J}$ be the collection of the connected components of ${\rm Max}(f)=\bigsqcup_{\lambda\in\R}{\rm Max}_\lambda(f)$. As each $C_j$ is contained in ${\rm Max}_{\lambda_j}(f)$ for some $\lambda_j\in\R$, we deduce that each ${\rm Max}_\lambda(f)$ is a union of connected components of ${\rm Max}(f)$. As the connected components of ${\rm Max}(f)$ constitute a locally finite family of $M$ and ${\rm Max}_\lambda(f)\cap{\rm Max}_\mu(f)=\varnothing$ if $\lambda\neq\mu$, the family $\{{\rm Max}_\lambda(f)\}_{\lambda\in\R}$ is locally finite in $M$, as required.
\end{proof}

We have seen in Example \ref{ptd} (i) that the set of local maxima of a (non-continuous) subanalytic function can be connected but not subanalytic. The following example shows that the set of local maxima ${\rm Max}(f)$ of a (non-continuous) subanalytic function $f$ can be connected and subanalytic whereas its family of level sets $\{{\rm Max}_\lambda(f)\}_{\lambda\in\R}$ is not locally finite (contrast this with Theorem \ref{main0} where the continuity of $f$ is assumed). 

\begin{example}\label{ex1}\em
Let ${\mathfrak F}$ be a weak category that contains algebraic intersections. Let $M\subset\R^n$ be a real analytic manifold and let $X\in{\mathfrak F}(M)$ of dimension $\geq1$. Let $p\in X$ be such that $\dim(X_p)=\dim(X)$ and assume that $p=0$ is the origin of $\R^n$. Define $Z_m:=\{\frac{1}{m}\leq\|\x\|<\frac{1}{m-1}\}$ for $m\geq2$ and $Z_1:=\{\|\x\|\geq1\}$, which are elements of ${\mathfrak F}(\R^n)$. Observe that $\bigsqcup_{m\geq1}Z_m=\R^n\setminus\{0\}$. Consider the $C$-semianalytic function $f:\R^n\to\R$ given by $f|_{Z_m}=m$ for $m\geq1$ and $f(p)=0$. We have ${\rm Max}(f)=\R^n\setminus\{0\}\in{\mathfrak F}(\R^n)$ and 
$$
{\rm Max}_\lambda(f)=\begin{cases}
Z_m&\text{if $\lambda=m$ for some $m\geq1$,}\\
\varnothing&\text{otherwise.}
\end{cases}
$$
The family $\{{\rm Max}_\lambda(f)\}_{\lambda\in\R}$ is not locally finite at the origin. Define $g:=f|_X$ and observe that ${\rm Max}(g)={\rm Max}(f)\cap X=X\setminus\{0\}\in{\mathfrak F}(M)$ and 
$$
{\rm Max}_\lambda(g)={\rm Max}_\lambda(f)\cap X=\begin{cases}
Z_m\cap X&\text{if $\lambda=m$ for some $m\geq1$,}\\
\varnothing&\text{otherwise.}
\end{cases}
$$ 
The family $\{{\rm Max}_\lambda(g)\}_{\lambda\in\R}$ is not locally finite at the origin.
\end{example}

We prove next Proposition \ref{sharp}.

\begin{proof}[of Proposition \em\ref{sharp}]
Define $Z_1:=\{\|\x\|\geq1\}$, $Z_m:=\{\frac{1}{m}\leq\|\x\|<\frac{1}{m-1}\}$ for $m\geq2$ and $Z:=\R^n\setminus\{0\}=\bigsqcup_{k\geq1}Z_k$. Consider the function
$$
g:Z\to\R,\ x\mapsto
\begin{cases}
k-1&\text{if $x\in Z_{2k-1}$,}\\
\frac{1}{\|x\|}-k&\text{if $x\in Z_{2k}$}
\end{cases}
$$
and observe that $Z_{2k}:=\{2k\geq\frac{1}{\|\x\|}>2k-1\}$. Thus, $g$ is continuous and its graph is $\Gamma:=\bigsqcup_{k\geq1}(\Gamma_{2k}\cup\Gamma_{2k-1})$, where
$$
\Gamma_{2k-1}:=Z_{2k-1}\times\{k-1\}\quad\text{and}\quad\Gamma_{2k}:=\Big\{2k\geq\frac{1}{\|\x\|}>2k-1,\z=\frac{1}{\|\x\|}-k\Big\},
$$
is a locally finite union of semialgebraic sets. As the family $\{\Gamma_{2k-1},\Gamma_{2k}\}_{k\geq1}$ is locally finite in $\R^n$, the graph of $g$ is a $C$-semianalytic subset of $\R^n$. Observe that $g$ is not $M$-compact, as it maps the compact set $\R^n\setminus Z_1$ onto $[0,+\infty)$. The sets 
\begin{align*}
&{\rm Max}(g)=\bigsqcup_{k\geq1}Z_{2k-1},\\
&{\rm Min}(g)=\bigsqcup_{k\geq1}(\cl(Z_{2k-1})\cap Z\setminus Z_{2k-1})\cup\Int(Z_{2k-1}),\\
&{\rm Extr}(g)=\bigsqcup_{k\geq1}\cl(Z_{2k-1})\cap Z
\end{align*}
are not subanalytic because the families of their connected components are not locally finite in $\R^n$.

Define $g_k:=\min\{\max\{g,k\},k+1\}-k:\R^n\setminus\{0\}\to[0,1]$, which is a semialgebraic function whose support is the punctured closed ball $\supp(g_k)=\ol{B}(0,\frac{1}{k})\setminus\{0\}$. The family $\{\supp(g_k)\}_{k\geq1}$ is locally finite in $\R^n\setminus\{0\}$) and $g:=\sum_{k\geq1}g_k$. Let $I\subset\R$ be an upperly bounded interval and let $\ell\geq1$ be such that $I\subset(-\infty,\ell]$. Then $g^{-1}(I)=(\sum_{k=1}^{2\ell+1}g_k)^{-1}(I)$ is a semialgebraic set. If $I\subset\R$ is not upperly bounded, pick $\ell\in I$ and write $J:=(-\infty,\ell)\cap I$. Then $I=J\cup[\ell,+\infty)$ and 
$$
g^{-1}(J\cup[\ell,+\infty))=\Big(\sum_{k=1}^{2\ell+1}g_k\Big)^{-1}(J)\cup g^{-1}([\ell,+\infty))=\Big(\sum_{k=1}^{2\ell+1}g_k\Big)^{-1}(J)\cup B\Big(0,\frac{1}{2\ell+1}\Big)\setminus\{0\}
$$
is a semialgebraic set. 

Denote $\cl_M(\cdot):=\cl(\cdot)\cap M$. Pick a point $p\in\cl_M(X)\setminus X$ and assume that $p=0$ is the origin. Consider the restriction $f:=g|_{X}:X\to\R$ whose graph belongs to ${\mathfrak F}(M\times\R)$ and the inverse image under $f$ of each interval of $\R$ belongs to ${\mathfrak F}(M)$. Consider also the restriction $h:=h|_{\cl_M(X)\setminus\{0\}}:\cl_M(X)\setminus\{0\}\to\R$. The subanalytic set $\cl_M(X)\setminus\{0\}$ is a closed subanalytic subset of $M\setminus\{0\}$. Then the family $\{{\rm Max}_\lambda(h)\}_{\lambda\in\R}$ is locally finite in $M\setminus\{0\}$ (analogously it happens with $\{{\rm Min}_\lambda(h)\}_{\lambda\in\R}$ and $\{{\rm Extr}_\lambda(h)\}_{\lambda\in\R}$). Denote ${\tt T}$ the operators ${\rm Max}$, ${\rm Min}$ or ${\rm Extr}$ and ${\tt T}_\lambda$ the operators ${\rm Max}_\lambda$, ${\rm Min}_\lambda$ or ${\rm Extr}_\lambda$ for $\lambda\in\R$. Let us prove: \em The family $\{{\tt T}_\lambda(f)\}_{\lambda\in\R}$ is not locally finite in $M$\em. 

Observe that ${\tt T}_\lambda(f)={\tt T}_\lambda(h)\cap X$ for each $\lambda\in\R$. Let $\{C_i\}_{i\in I}$ be the collection of the connected components of ${\tt T}(f)$. As $\cl_M(X)\setminus\{0\}$ is closed in the real analytic manifold $M\setminus\{0\}$ and $h:\cl_M(X)\setminus\{0\}\to\R$ is a continuous subanalytic function, $\{{\tt T}_\lambda(f)={\tt T}_\lambda(h)\cap X\}_{\lambda\in\R}$ is by Theorem \ref{main1} a locally finite family in $M\setminus\{0\}$, so it is a countable family. Thus, $f({\tt T}(f))=\bigsqcup_{\lambda\in\R}f({\tt T}_\lambda(f))$ is a countable subset of $\R$. Let $C$ be a connected component of ${\tt T}(f)$. Then $f(C)$ is a connected subset of $\R$ contained in the countable set $f({\tt T}(f))$, so it is a singleton, that is, $f(C)=\{\lambda\}$ for some $\lambda\in\R$ and $C\subset{\tt T}_\lambda(f)={\tt T}(f)\cap\{f=\lambda\}$. Let $\{C_i\}_{i\in I}$ be the collection of the connected components of ${\tt T}(f)$. As $\bigsqcup_{i\in I}C_i={\tt T}(f)=\bigsqcup_{\lambda\in\R}{\tt T}_\lambda(f)$, each set ${\tt T}_{\lambda}(f)$ is a union of some $C_i$. 

Let us restrict to those values $\lambda$ that are positive integers and use the letter $k$ to denote them instead of $\lambda$. We have $\bigsqcup_{k\geq1}{\tt T}_k(f)$ is a union of some of the sets $C_i$. If ${\tt T}(f)$ is a subanalytic subset of $M$, then the collection $\{C_i\}_{i\in I}$ is a locally finite family of subanalytic subsets of $M$. Thus, the family $\{{\tt T}_k(f)\}_{k\geq1}$ is a locally finite family of subanalytic subsets of $M$. As $\Int(Z_{2k+1})\cap X\subset{\tt T}_k(f)$ for each $k\geq1$, also the family $\{\Int(Z_{2k+1})\cap X\}_{k\geq1}$ of open subanalytic subsets of $X$ is locally finite in $M$, but this is a contradiction because $\Int(Z_{2k+1})\cap X\neq\varnothing$ for $k$ large enough and $\{\Int(Z_{2k+1})\cap X)\}_{k\geq1}$ is not locally finite at the origin. Consequently, ${\tt T}(f)$ is not a subanalytic subset of $M$, as required.
\end{proof}

\subsection{Strict local maxima of subanalytic functions}
As a consequence of Theorem \ref{main1} we have the following.

\begin{cor}[(Strict local maxima)]\label{main3}
Let $f:X\to\R$ be an $M$-compact subanalytic function on a subanalytic subset $X$ of a real analytic manifold $M$. Let $T\subset X$ be the set of strict local maxima of $f$. Then $T$ is a discrete subset of $M$.
\end{cor}
\begin{proof}
Let $x_0\in T$ and let $\lambda_0=f(x_0)$. It holds $x_0\in{\rm Max}_{\lambda_0}(f)$ and as it is a strict local maxima, it is an isolated point, so $x_0$ is a connected component of ${\rm Max}_{\lambda_0}(f)$. As the family $\{{\rm Max}_\lambda(f)\}_{\lambda\in\R}$ is by Theorem \ref{main1} locally finite and the family of the connected components of each ${\rm Max}_\lambda(f)$ is locally finite, we conclude that $T$ is a locally finite union of points. This means that $T$ is a discrete subset of $M$, as required.
\end{proof}
\begin{remark}\em
Let $f:M\to\R$ be an analytic function on a real analytic manifold $M\subset\R^n$ and let $H$ be the set of critical points $x\in M$ of $f$ such that the Hessian of $f$ at $x$ is negative definite. Then $f$ is $M$-compact and $H$ is contained in the set $T$ of strict local maxima of $f$. In particular $H$ is a discrete subset of $M$. 
\end{remark}

Again the condition that $f$ is $M$-compact in Corollary \ref{main3} is not superfluous.
\begin{example}\em
Let $X:=(0,1]\subset\R$ and let $f:X\to\R$ be the continuous function whose graph is the polygonal through the points $\{(\frac{1}{2k-1},2k-3),(\frac{1}{2k},2k):\ k\geq1\}$ ordered in terms of their first coordinates, that is,
$$
f(x):=\begin{cases}
-6k(2k-1)x+8k-3&\text{if $\frac{1}{2k}\leq x\leq\frac{1}{2k-1}$,}\\
2k(2k+1)x-1&\text{if $\frac{1}{2k+1}\leq x\leq\frac{1}{2k}$}.
\end{cases}
$$
The graph of $f$ is a $C$-semianalytic set because it is the union of the two locally finite families of basic $C$-semianalytic sets 
\begin{align*}
&\Big\{\frac{1}{2k}\leq\x\leq\frac{1}{2k-1},\y=-6k(2k-1)\x+8k-3\Big\}_{k\geq1},\\
&\Big\{\frac{1}{2k+1}\leq\x\leq\frac{1}{2k},\y=2k(2k+1)\x-1\Big\}_{k\geq1}
\end{align*}
of $\R^2$. Thus, $f$ is $C$-semianalytic, but $f$ is not $M$-compact because $f((0,1])=(0,+\infty)$. It holds ${\rm Max}(f)=\{\frac{1}{2k}:\ k\geq1\}$, which is not a discrete subset of $M$.
\end{example}

\subsection{Non-openness points of continuous subanalytic functions} 
Open maps play an important role in analysis. Some classic theorems state the openness of various regular maps, for instance, the Banach openness principle concerning linear operators (in functional analysis), and the open map theorem dealing with holomorphic functions (in complex analysis). To decide if a map is open is in general a difficult question. This problem was studied in \cite{ct} and later a complete answer was provided in \cite{gr} for Nash maps $f:\R^n\to\R^n$. The previous result was extended in \cite{hsch} to analytic maps.

Recall that a function $f$ from a topological space $X$ into a topological space $Y$ is \em open \em if it maps open sets onto open sets. We say that $f$ is open at a point $x\in X$ if $f(x)$ belongs to $\Int(f(U))$ for each open neighborhood $U\subset X$ of $x$. Plainly, $f$ is open if and only if it is open at every point $x\in X$. In \cite{bpw} it is proved the following result.

\begin{cor}[(Non-openness points, {\cite{bpw}})]
If $X$ is a locally connected space and a function $f:X\to\R$ is continuous, then the set of points of local extrema of $f$ coincides with the set of its non-openness points. 
\end{cor}

Let $M$ be a real analytic manifold. As each subanalytic set $X\subset M$ is locally connected \cite{bm}, Lemma \ref{main2} applies and we conclude that both the sets of non-openness points ${\rm NOp}(f)={\rm Extr}(f)$ and of openness points ${\rm Op}(f)=X\setminus{\rm Extr}(f)$ of a continuous subanalytic function $f:X\to\R$ on a closed semianalytic subset $X$ of $M$ are subanalytic subsets of $M$. More generally, if ${\mathfrak F}$ is a weak category of subanalytic sets and $f:X\to\R$ is a continuous subanalytic function on a closed set $X\in{\mathfrak F}(M)$ such that the inverse images of the intervals of $\R$ belong to ${\mathfrak F}(M)$, then ${\rm Op}(f)=X\setminus{\rm Extr}(f)\in{\mathfrak F}(M)$ and ${\rm NOp}(f)\in{\mathfrak F}(M)$. In particular, the previous applies if ${\mathfrak F}$ is either the category of semianalytic or $C$-semianalytic subsets of $M$. If $f$ is a non-constant analytic function and $X$ is an irreducible $C$-analytic set \cite{fe}, then $\dim({\rm Extr}(f))\leq\dim(X)-1$ and $\Int({\rm Op}(f))$ is a dense open $C$-semianalytic subset of the ($C$-semianalytic) set of points of maximal dimension of $X$.

If $X\in{\mathfrak F}(M)$ is not closed in $M$, there exists by Proposition \ref{sharp} a continuous subanalytic function $f:X\to\R$ whose graph is an ${\mathfrak F}$-set, the inverse images under $f$ of the intervals of $\R$ belong to ${\mathfrak F}(M)$ and ${\rm Extr}(f)\not\in{\mathfrak F}(M)$, because it is not even a subanalytic subset of $M$. Thus, ${\rm Op}(f)=X\setminus{\rm Extr}(f)\not\in{\mathfrak F}(M)$ (and it is not even a subanalytic subset of $M$).

\section{Locally normal crossings real analytic functions}\label{s4}

As usual we denote $\an(M)$ the ring of real analytic functions on a real analytic manifold $M$. In the analytic case, the local extrema of an analytic function $f:M\to\R$ are contained in its set of critical points. A point $x\in M$ is \em critical for $f$ \em if there exists a chart $\varphi:M\to\R^m$ such that $\varphi(x)=0$ and $\nabla(f\circ\varphi^{-1})(0):=(\frac{\partial(f\circ\varphi^{-1})}{\partial\x_1}(0),\ldots,\frac{\partial(f\circ\varphi^{-1})}{\partial\x_m}(0))=0$. As one can expect the previous definition does not depend on the chosen chart. The set of critical points of $f$ in $M$ will be denoted ${\tt C}(f)$. It is easily checked that it is a closed subset of $M$.

\begin{remarks}\label{31}\em
(i) If we restrict our target to analytic functions on real analytic manifolds, we can even restrict to the case of analytic functions on open subsets of $\R^n$ using tubular neighborhoods. Let $(\Omega,\rho)$ be an analytic tubular neighborhood of $M$ in $\R^n$. If $x_0\in M$ is a local maximum of $f$, then the points $y\in\rho^{-1}(x_0)$ are local maxima of $f\circ\rho$. As $f\circ\rho$ is constant on each fiber of $\rho$, each extremal point $y_0$ of $f\circ\rho$ provides a extremal point $x_0:=\rho(y_0)$ of $f$. In fact, the set ${\rm Max}(f)$ of local maxima of $f$ coincides with the intersection ${\rm Max}(f\circ\rho)\cap M$. Thus, we could focus on analytic functions defined on open subsets of $\R^n$. 

(ii) If $f:\Omega\to\R$ is an analytic function on an open set $\Omega\subset\R^n$, the set of critical points of $f$ is ${\tt C}(f)=\{\frac{\partial f}{\partial\x_1}=0,\ldots,\frac{\partial f}{\partial\x_n}=0\}$.

(iii) If $X$ is a connected real analytic manifold (or more generally a pure dimensional irreducible $C$-semianalytic set \cite{fe}) and $f$ is a non-constant analytic function on $X$, then ${\rm Max}(f)\cap{\rm Min}(f)=\varnothing$ by Remark \ref{imext} (ii) and the identity principle.
\end{remarks}

A particular case of (real) analytic functions for which it is easy to characterize local extrema is that of \em locally normal crossings \em analytic functions.

\begin{defn}[(Local normal crossings, {\cite[Def. 4.3]{bm}})]\em
Let $f\in\an(M)$ be an analytic function that is not constant on any connected component of $M$. We say that $f$ is \em locally normal crossings \em if each point $x_0$ of $M$ admits a coordinate neighborhood $U$, with coordinates $\x:=(\x_1,\ldots,\x_m)$, such that $f(x)=x_1^{\alpha_1}\cdots x_m^{\alpha_m}h(x)$ for each $x\in U$, where $h\in\an(U)$, $h$ vanishes nowhere in $U$ and each $\alpha_i\geq0$.
\end{defn}

The following result allows to improve the description of locally normal crossings analytic functions.

\begin{lem}\label{change}
Let $W\subset\R^m$ be an open neighborhood of the origin and let $g\in\an(W)$ be an analytic function that vanishes nowhere on $W$. Define $h:=\x_m^\alpha g$ for some $\alpha\geq1$. Then there exists an analytic change of coordinates 
$$
\psi:U\to V\subset W,\ (\y_1,\ldots,\y_m)\mapsto(\y_1,\ldots,\y_{m-1},\y_m u_m)
$$ 
on a small open neighborhood $V$ of the origin, where $u_m\in\an(U)$ vanishes nowhere on $U$, such that $(h\circ\psi)(\y_1,\ldots,\y_m)=\pm\y_m^\alpha$ and $\psi$ keeps invariant the coordinate hyperplanes.
\end{lem}
\begin{proof}
Let $a:=g(0)$ and assume that $a>0$. We write $b:=\frac{1}{\sqrt[\alpha]{a}}$. Consider the analytic equation 
$$
F(\y_1,\ldots,\y_m,\t):=(b+\t)^\alpha g(\y_1,\ldots,\y_{m-1},\y_m(b+\t))-1.
$$
Observe that $F(0,0)=b^\alpha a-1=0$ and
$$
\frac{\partial F}{\partial\t}=\alpha(b+\t)^{\alpha-1}g(\y_1,\ldots,\y_{m-1},\y_m(b+\t))-(b+\t)^\alpha\frac{\partial g}{\partial\y_m}(\y_1,\ldots,\y_{m-1},\y_m(b+\t))\y_m.
$$
Thus, $\frac{\partial F}{\partial\t}(0,0)=\alpha b^{\alpha-1}a=\frac{\alpha}{b}\neq0$. By the implicit function theorem there exists $\xi\in\R\{\y\}$ such that $\xi(0)=0$ and $F(\y,\xi(\y))=0$. Consider the local change of coordinates 
$$
\psi:(\y_1,\ldots,\y_m)\mapsto(\y_1,\ldots,\y_{m-1},\y_m(b+\xi))
$$ 
around the origin. We have
$$
(h\circ\psi)(\y_1,\ldots,\y_m)=\y_m^{\alpha}(b+\xi)^\alpha g(\y_1,\ldots,\y_{m-1},\y_m(b+\xi))=\y_m^{\alpha}.
$$
In addition, $\psi$ keeps invariant the coordinate hyperplanes, as required.
\end{proof}

Let us study next from the local point of view the sets of critical points and local extrema of locally normal crossings analytic functions. 

\begin{lem}\label{localdes}
Let $W\subset\R^m$ be a connected open set, let $h\in\an(W)$ be an analytic function that vanish nowhere on $W$ and let $\alpha_1,\ldots,\alpha_r\geq2$ be positive integers, where $1\leq r\leq m$. Define $f:=\x_1^{\alpha_1}\cdots\x_r^{\alpha_r}\x_{r+1}\cdots\x_dh\in\an(W)$, where $r\leq d\leq m$, and assume that $\alpha_1,\ldots,\alpha_\ell$ are even whereas $\alpha_{\ell+1},\ldots,\alpha_r$ are odd. Then there exists an open neighborhood $U\subset W$ of $X:=\bigcup_{i=1}^r\{\x_i=0\}$ such that set of critical points of $f|_U$ is 
$$
\bigcup_{r+1\leq i<j\leq d}\{\x_i=0,\x_j=0\}\cup\bigcup_{i=1}^r\{\x_i=0\}.
$$
In addition, if $g$ takes (only) positive values on $W$, the set of local maxima of $f$ is 
$$
\bigcup_{i=1}^\ell\{\x_i=0\}\cap\{\x_{\ell+1}\cdots\x_d<0\}
$$ 
whereas the set of local minima of $f$ is
$$
\bigcup_{i=1}^\ell\{\x_i=0\}\cap\{\x_{\ell+1}\cdots\x_d>0\}.
$$
\end{lem}
\begin{proof}
Pick a point $x\in X$ and assume that $x\in\{\x_i=0\}$ exactly for the indices $i=1,\ldots,s\leq\ell$, $i=\ell+1,\ldots,k\leq r$ and $i=r+1,\ldots,e\leq d$ (in order to clarify, note that $x\not\in\{\x_i=0\}$ for the indices $i=s+1,\ldots,\ell$, $i=k+1,\ldots,r$ and $i=e+1,\ldots,d$). Define 
$$
g_x:=(\x_{s+1}^{\alpha_{s+1}}\cdots\x_\ell^{\alpha_\ell})(\x_{k+1}^{\alpha_{k+1}}\cdots\x_r^{\alpha_r})(\x_{e+1}\cdots\x_d)h.
$$
By Lemma \ref{change} for each point there exist an open neighborhood $U^x\subset W$ of $x$, an open neighborhood $V_x\subset\R^m$ of the origin and a change of coordinates $\psi_x:V_x\to U^x$ that maps the origin to $x$, keeps invariant the coordinate hyperplanes through $x$ and satisfies 
$$
f_x:=(f\circ\psi_x)(\y_1,\ldots,\y_m)=\pm(\y_1^{\alpha_1}\cdots\y_s^{\alpha_s})(\y_{\ell+1}^{\alpha_{\ell+1}}\cdots\y_k^{\alpha_k})(\y_{r+1}\cdots\y_e).
$$
The sign corresponding to $f_x$ is $+$ if $g_x(x)>0$ and $-$ if $g_x(x)<0$. Define $U:=\bigcup_{x\in X}U^x\subset W$. The set of critical points of $f|_U$ coincides with the union of the critical points of $f_x$ for each $x\in X$. We have
$$
\frac{\partial f_x}{\partial\y_i}=
\begin{cases}
\alpha_i\frac{f_x}{\y_i}&\text{if $i=1,\ldots,s$ or $i=\ell+1,\ldots,k$},\\
\frac{f_x}{\y_i}&\text{if $i=r+1,\ldots,e$},\\
0&\text{if $i=s+1,\ldots,\ell$, $i=k+1,\ldots,r$ or $i=e+1,\ldots,m$},
\end{cases}
$$
so the set of critical points of $f_x$ is 
$$
\Big\{\frac{\partial f_x}{\partial\y_1}=0,\ldots,\frac{\partial f_x}{\partial\y_m}=0\Big\}=\bigcup_{i=1}^s\{\y_i=0\}\cup\bigcup_{i=\ell+1}^k\{\y_i=0\}\cup\bigcup_{r+1\leq i<j\leq e}\{\y_i=0,\y_j=0\}.
$$
Thus, the set of critical points of $f|_U$ is
$$
\bigcup_{r+1\leq i<j\leq d}\{\x_i=0,\x_j=0\}\cup\bigcup_{i=1}^r\{\x_i=0\}.
$$
Assume $x\in\bigcup_{i=\ell+1}^d\{\x_i=0\}$. We can suppose that either $x\in\{\x_{\ell+1}=0\}$ or $x\in\{\x_{r+1}=0\}$. If $x\in\{\x_{\ell+1}=0\}$, we write $\lambda_\t:=(\lambda_1,\ldots,\lambda_\ell,\t,\lambda_{\ell+2},\ldots,\lambda_m)$ and
$$
f(x+\lambda_\t)=\t^{\alpha_{\ell+1}}\prod_{i=1}^\ell(x_i+\lambda_i)^{\alpha_i}\prod_{i=\ell+2}^r(x_i+\lambda_i)^{\alpha_i}\prod_{i=r+1}^d(x_i+\lambda_i)\cdot h(x+\lambda_\t).
$$
Pick $\lambda_1,\ldots,\lambda_\ell,\lambda_{\ell+2},\ldots,\lambda_m\in\R$ small enough such that
$$
a:=\prod_{i=1}^\ell(x_i+\lambda_i)^{\alpha_i}\prod_{i=\ell+2}^r(x_i+\lambda_i)^{\alpha_i}\prod_{i=r+1}^d(x_i+\lambda_i)\neq0
$$
and the sign of $h(x+\lambda_\t)$ coincides with that of $h(x)$ around $\t=0$. Thus, $f(x+\lambda_\t)=a\t^{\alpha_{\ell+1}}h(x+\lambda_\t)$, which changes sign around $\t=0$. As $f(x)=0$, this means that $f$ does not have a local extremum at $x$. If $x\in\{\x_{r+1}=0\}$, the discussion is analogous. This means that the set of local extrema of $f|_U$ is contained in 
$$
\Big(\bigcup_{r+1\leq i<j\leq d}\{\x_i=0,\x_j=0\}\cup\bigcup_{i=1}^r\{\x_i=0\}\Big)\setminus\bigcup_{i=\ell+1}^d\{\x_i=0\}=\bigcup_{i=1}^\ell\{\x_i=0\}\cap\{\x_{\ell+1}\cdots\x_d\neq0\}.
$$
Assume that $g$ only takes positive values on $W$ and let $x\in\bigcup_{i=1}^\ell\{\x_i=0\}\cap\{\x_{\ell+1}\cdots\x_d\neq0\}$. Let us check that $x$ is either a local maximum or a local minimum of $f$. If $x\in\{\x_{\ell+1}\cdots\x_d<0\}$, we may assume that $x\in\{\x_i=0\}$ exactly if $i=1,\ldots,s\leq\ell$ and
$$
(f\circ\psi_x)=-\y_1^{\alpha_1}\cdots\y_s^{\alpha_s}.
$$
As the exponents $\alpha_1,\ldots,\alpha_s$ are all even, the point $x$ is a local maximum of $f$. Analogously, if $x\in\{\x_{\ell+1}\cdots\x_d>0\}$, the point $x$ is a local minimum of $f$, as required.
\end{proof}

We use the previous lemma to describe the sets of local extrema of locally normal crossings analytic functions.

\begin{thm}\label{easydesc}
Let $M$ be a real analytic manifold and let $f\in\an(M)$ be a locally normal crossing real analytic function. Then, there exist real analytic functions $h,g\in\an(M)$ such that ${\rm Max}_0(f)=\{h=0\}\cap\{g<0\}$ and ${\rm Min}_0(f)=\{h=0\}\cap\{g>0\}$. 
\end{thm}
\begin{proof}
Let $\{Z_i\}_{i\in I}$ be the collection of the irreducible components of the coherent hypersurface $Z$. Observe that each $Z_i$ is a hypersurface of $M$. For each $i\in I$ denote $m_i$ the multiplicity of $f$ along the hypersurface $Z_i$ (recall that the multiplicity along a hypersurface is a discrete valuation \cite[p.300]{ad}). Denote 
\begin{align*}
&I_+:=\{i\in I:\ m_i \text{ is even}\},\\
&I_-:=\{i\in I:\ m_i \text{ is odd}\}.
\end{align*}
Let $Z_+:=\bigcup_{i\in I^+}Z_i$ and $Z_-:=\bigcup_{i\in I_-}Z_i$. Write $m_i:=2k_i$ for each $i\in I_+$ and let $h_{x,i}$ be a local generator of the ideal $\Jj(Z_{i,x})$ of germs at $x$ of analytic functions on $M$ vanishing identically on $Z_{i,x}$. Consider the coherent sheaf of ideals
$$
\Ff_x:=\begin{cases}
\prod_{i\in I^+:\ x\in Z_i}h_{x,i}^{k_i}\an_{M,x}&\text{if $x\in Z_+$,}\\
\an_{M,x}&\text{otherwise}.
\end{cases}
$$
By \cite{co} there exist finitely many global sections $h_1,\ldots,h_r\in\an(M)$ that generates the sheaf of ideals $\Ff$. Observe that $h:=h_1^2+\cdots+h_r^2$ divides $f$, that is, there exists an analytic function $g\in\an(M)$ such that $f=gh$. Note that $\{g=0\}=Z_+$ and $g$ changes sign at each point $x\in Z_+$. By Lemma \ref{localdes} we conclude
\begin{align*}
{\rm Max}(f)\cap U&={\rm Max}_0(f)=\{h=0\}\cap\{g<0\},\\
{\rm Min}(f)\cap U&={\rm Min}_0(f)=\{h=0\}\cap\{g>0\},
\end{align*} 
as required.
\end{proof}

A main tool in the study of continuous subanalytic functions on closed subanalytic subsets of a real analytic manifold $M$ is local uniformization \cite[\S4]{bm}. If we combine \cite[Thm. 0.1, Cor. 4.9, Lem. 5.3]{bm} we have the following result.

\begin{thm}[(Local uniformization, {\cite{bm}})]\label{ut2}
Let $X\subset M$ be a closed subanalytic subset and let $f:X\to\R$ be a continuous subanalytic function. For each $\lambda\in\R$ there exist a real analytic manifold $N_\lambda$ and a proper surjective real analytic map $\pi_\lambda:N_\lambda\to X\subset M$ such that $(f-\lambda)\circ\pi_\lambda$ is an analytic map on $N_\lambda$ that is locally normal crossings.
\end{thm}

The previous result allows us to provide an alternative description, that does not involve closures, of the level sets ${\rm Max}_\lambda(f)$ of a continuous subanalytic function $f:X\to\R$ on a closed subanalytic subset $X$ of a real analytic manifold $M$. Of course this description preserves the subanalyticity in $M$ of each set ${\rm Max}_\lambda(f)$.

\begin{cor}\label{altdesc}
Let $X\subset M$ be a closed subanalytic subset and let $f:X\to\R$ be a continuous subanalytic function. Fix $\lambda\in\R$ and let $N_\lambda$ be a real analytic manifold. Let $\pi_\lambda:N_\lambda\to X\subset M$ be a proper surjective real analytic map such that $(f-\lambda)\circ\pi_\lambda$ is an analytic map on $N_\lambda$ that is locally normal crossings. Then for each $\lambda\in\R$ there exist analytic functions $h_\lambda,g_\lambda\in\an(N_\lambda)$ such that
\begin{multline}\label{altdesceq}
{\rm Max}_\lambda(f)=\{f=\lambda\}\setminus\pi_\lambda(\{f\circ\pi_\lambda=\lambda\}\setminus{\rm Max}(f\circ\pi_\lambda))\\
=\{f=\lambda\}\setminus\pi_\lambda(\{f\circ\pi_\lambda=\lambda\}\setminus\{h_\lambda=0,g_\lambda<0\}).
\end{multline}
\end{cor}
\begin{proof}
By Lemma \ref{equiv} a point $x\in Z_\lambda:=\{f=\lambda\}$ is a local maximum if and only if each point $y\in\pi_\lambda^{-1}(x)$ belongs to ${\rm Max}_\lambda(f\circ\pi_\lambda)$. Thus, 
$$
Z_\lambda\setminus{\rm Max}_\lambda(f)=\pi_\lambda(\{f\circ\pi_\lambda=\lambda\}\setminus{\rm Max}(f\circ\pi_\lambda)),
$$ 
or equivalently, 
$$
{\rm Max}_\lambda(f)=\{f=\lambda\}\setminus\pi_\lambda(\{f\circ\pi_\lambda=\lambda\}\setminus{\rm Max}(f\circ\pi_\lambda)).
$$
By Theorem \ref{easydesc} there exist analytic functions $h_\lambda,g_\lambda\in\an(N_\lambda)$ such that ${\rm Max}(f\circ\pi_\lambda)=\{h_\lambda=0,g_\lambda<0\}$ and the last equality in \eqref{altdesceq}, as required.
\end{proof}


\begin{thebibliography}{BPW}

\bibitem[ABF]{abf} F. Acquistapace, F. Broglia, J.F. Fernando: On globally defined semianalytic sets. \em Math. Ann. \em {\bf366} (2016), no. 1, 613--654.

\bibitem[AD]{ad} F. Acquistapace, A. D\'\i az-Cano: Divisors in global analytic sets. {\em J. Eur. Math. Soc.} (JEMS) 13 (2011), no. 2, 297--307.

\bibitem[BPW]{bpw} M. Balcerzak, M. Pop\l awski, J. W\'odka: Local extrema and nonopenness points of continuous functions. {\em Amer. Math. Monthly} {\bf124} (2017), no. 5, 436--443.

\bibitem[BGZ]{bgz} A. Barone-Netto, G. Gorni, G. Zampieri: Local extrema of analytic functions. {\em NoDEA Nonlinear Differential Equations Appl.} {\bf3} (1996), no. 3, 287--303.

\bibitem[BM1]{bm} E. Bierstone, P.D. Milman: Semianalytic and subanalytic sets. \em Inst. Hautes \'Etudes Sci. Publ. Math. \em {\bf67} (1988), 5--42.

\bibitem[BM2]{bm2} E. Bierstone, P.D. Milman: Subanalytic geometry. Model theory, algebra, and geometry, 151--172, \em Math. Sci. Res. Inst. Publ.\em, {\bf39}, Cambridge Univ. Press, Cambridge (2000).

\bibitem[CV]{cv} B. Calvert, M.K. Vamanamurthy: Local and global extrema for functions of several variables. {\em J. Austral. Math. Soc. Ser. A} {\bf29} (1980), no. 3, 362--368. 

\bibitem[CT]{ct} P.T. Church, J.G. Timourian: Real analytic open maps. {\em Pacific J. Math.} {\bf50} (1974), 37--42.

\bibitem[C]{co} S. Coen: Sul rango dei fasci coerenti. {\em Boll. Un. Mat. Ital.} (3) {\bf22} (1967) 373--382.

\bibitem[DD]{dd} J. Denef, L. van den Dries: $p$-adic and real subanalytic sets. \em Ann. of Math. \em (2) {\bf128} (1988), no. 1, 79--138.

\bibitem[D]{de} Z. Denkowska: La continuit\'e de la section d'un ensemble semi-analytique et compact. \em Ann. Polon. Math. \em {\bf37} (1980), no. 3, 231--242.

\bibitem[vdD]{vdD1} L. Van den Dries: Tame topology and o-minimal structures. {\em London Mathematical Society Lecture Note Series}, {\bf248}. Cambridge University Press, Cambridge (1998).

\bibitem[vdDM]{dm} L. van den Dries, C. Miller: Geometric categories and o-minimal structures. {\em Duke Math. J.} {\bf84} (1996), no. 2, 497--540.

\bibitem[FL]{fl} A. Fedeli, A. Le Donne: On metric spaces and local extrema. {\em Topology Appl.} {\bf156} (2009), no. 13, 2196--2199. 

\bibitem[F]{fe} J.F. Fernando: On the irreducible components of globally defined semianalytic sets.
{\em Math. Z.} {\bf283} (2016), no. 3-4, 1071--1109.

\bibitem[G]{ga} A. Gabrielov: Projections of semi-analytic sets. \em Functional Anal. Appl.\em, {\bf 2} (1968), no. 4, 282--291. 

\bibitem[Ga]{gal} M. Galbiati: Stratifications et ensemble de non-coh\'erence d'un espace analytique r\'eel. \em Invent. Math\em. {\bf34} (1976), no. 2, 113--128.

\bibitem[GR]{gr} J.M. Gamboa, F. Ronga: On open real polynomial maps. {\em J. Pure Appl. Algebra} {\bf110} (1996), no. 3, 297--304.

\bibitem[Ha1]{h1} R.M. Hardt: Homology and images of semianalytic sets. \em Bull. Amer. Math. Soc. \em {\bf80} (1974), 675--678.

\bibitem[Ha2]{h2} R.M. Hardt: Homology theory for real analytic and semianalytic sets. \em Ann. Scuola Norm. Sup. Pisa Cl. Sci. \em (4) {\bf2} (1975), no. 1, 107--148. 

\bibitem[H1]{hi1} H. Hironaka: Introduction aux ensembles sous-analytiques. \em R\'edig\'e par Andr\'e Hirschowitz et Patrick Le Barz. Singularit\'es \`a Carg\`ese \em (Rencontre Singularit\'es en G\'eom. Anal., Inst. \'Etudes Sci., Carg\`ese, 1972), pp. 13--20. \em Asterisque\em, {\bf7} et {\bf8}, Soc. Math. France, Paris (1973).

\bibitem[H2]{hi2} H. Hironaka: Subanalytic sets. \em Number theory, algebraic geometry and commutative algebra\em, in honor of Yasuo Akizuki, pp. 453--493. Kinokuniya, Tokyo (1973).

\bibitem[H3]{hi3} H. Hironaka: Introduction to real-analytic sets and real-analytic maps. \em Quaderni dei Gruppi di Ricerca Matematica del Consiglio Nazionale delle Ricerche\em. Istituto Matematico ``L. Tonelli'' dell'Universit\`a di Pisa, Pisa (1973).

\bibitem[H4]{hi4} H. Hironaka: Stratification and flatness. Real and complex singularities (\em Proc. Ninth Nordic Summer School/NAVF Sympos. Math.\em, Oslo, 1976), pp. 199--265. Sijthoff and Noordhoff, Alphen aan den Rijn (1977).

\bibitem[Hi]{hsch} M.W. Hirsch: Jacobians and branch points of real analytic open maps. {\em Aequationes Math.} {\bf63} (2002), no. 1--2, 76--80.

\bibitem[Ko]{ko} B. Kocel-Cynk: Finite decidability of subanalytic functions. Effective methods in algebraic and analytic geometry (Bielsko-Bia\l a, 1997). {\em Univ. Iagel. Acta Math.} {\bf37} (1999), 151--154.

\bibitem[K]{k} K. Kurdyka: Des applications du th\'eor\`eme de Puiseux dans la th\'eorie des ensembles semi-analytiques dans $\R^2$. \em Univ. Iagel. Acta Math. \em {\bf26} (1987), 105--113.

\bibitem[\L1]{l} S. \L ojasiewicz: Ensembles semi-analytiques, Cours Facult\'e des Sciences d'Orsay, \em Mimeographi\'e I.H.E.S.\em, Bures-sur-Yvette, July (1965). {\tt http://perso.univ-rennes1.fr/michel.coste/Lojasiewicz.pdf}

\bibitem[\L2]{l1} S. \L ojasiewicz, Triangulation of semi-analytic sets. \em Ann. Scuola Norm. Sup. Pisa \em (3) {\bf18} (1964) 449--474.

\bibitem[N]{n} R. Narasimhan: Introduction to the theory of analytic spaces, {\em Lecture Notes in Math.} {\bf 25}, Springer-Verlag, Berlin-New York (1966).

\bibitem[P]{p} A. Parusi\'nski: Lipschitz properties of semi-analytic sets. Ann. Inst. Fourier (Grenoble) {\bf38} (1988), no. 4, 189--213.

\bibitem[Pa]{pa} W. Paw\l ucki: Sur les points r\'eguliers d'un ensemble semi-analytique. \em Bull. Polish Acad. Sci. Math. \em {\bf32} (1984), no. 9--10, 549--553.

\bibitem[Sh]{sh} M. Shiota: Geometry of subanalytic and semialgebraic sets. {\em Progress in Mathematics}, {\bf150}. Birkh\"auser Boston, Inc., Boston, MA (1997). 

\bibitem[S]{s} J. Stasica: Smooth points of a semialgebraic set. \em Ann. Polon. Math. \em {\bf82} (2003), no. 2, 149--153. 

\bibitem[W]{w} A.J. Wilkie: Lectures on elimination theory for semialgebraic and subanalytic sets. O-minimality and diophantine geometry, 159--192, {\em London Math. Soc. Lecture Note Ser.}, {\bf421}, Cambridge Univ. Press, Cambridge (2015).

\end{thebibliography}
\end{document}